\documentclass{article}

\usepackage{arxiv}

\usepackage[utf8]{inputenc}
\usepackage[T1]{fontenc}
\usepackage[numbers,sort&compress]{natbib}
\usepackage{url}
\usepackage{booktabs}
\usepackage{amsfonts}
\usepackage{amssymb}
\usepackage{nicefrac}
\usepackage{microtype}
\usepackage{graphicx}
\usepackage{amsthm}
\usepackage{amsmath}
\usepackage{hyperref}

\numberwithin{equation}{section}

\theoremstyle{plain}
\newtheorem{theorem}{Theorem}[section]

\newtheorem{lemma}[theorem]{Lemma}

\newtheorem{remark}[theorem]{Remark}

\newtheorem{problem}[theorem]{Problem}

\theoremstyle{remark}

\newtheorem{definition}[theorem]{Definition}

\def\R{\mathbb{R}}

\def\N{\mathbb{N}}
\def\G{\mathbb{G}}
\def\T{\mathbb{T}}

\DeclareMathOperator{\Ex}{\mathbb{E}}
\renewcommand{\Pr}{\mathbb{P}}
\newcommand{\Pm}{P}
\newcommand{\Pmhat}{\widehat{P}}
\newcommand{\Tmhat}{\widehat{T}}

\newcommand{\muhat}{\widehat{\mu}}

\def\sF{\mathcal{F}}
\def\sG{\mathcal{G}}\def\sH{\mathcal{H}}\def\sI{\mathcal{I}}

\def\sN{\mathcal{N}}

\def\bX{{\bf X}}

\def\eps{\varepsilon}
\def\Ind{{\bf 1}}

\DeclareMathOperator*{\argmin}{arg\,min}

\title{Robust high-dimensional Gaussian and bootstrap approximations for trimmed sample means}

\author{ \href{https://orcid.org/0000-0001-6188-299X}{\includegraphics[scale=0.06]{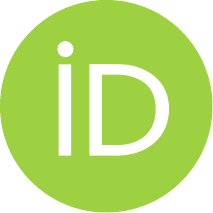}\hspace{1mm}Lucas Resende}\\
CREST, ENSAE, Institut Polytechnique de Paris\\
Palaiseau, France \\
\texttt{lucas.resende@ensae.fr} \\
}

\renewcommand{\headeright}{}
\renewcommand{\undertitle}{}
\renewcommand{\shorttitle}{Robust high-dimensional Gaussian and bootstrap approximations for trimmed sample means}

\hypersetup{
pdftitle={Robust high-dimensional Gaussian and bootstrap approximations for trimmed sample means},
pdfsubject={stat.ME, math.ST},
pdfauthor={Lucas Resende},
pdfkeywords={Gaussian approximation, bootstrap, high dimensionality, robustness},
}

\begin{document}
\maketitle

\begin{abstract}
Robust mean estimation has largely focused on concentration guarantees under heavy tails and contamination. We study robustness from a different perspective: high-dimensional Gaussian and bootstrap approximations. We show that trimmed sample means admit Gaussian and bootstrap approximations under finite (p)-th moment assumptions, even in high-dimensional regimes and in the presence of adversarial contamination. Our bounds recover, up to the dependence on the moment parameter, the rates available for the empirical mean under light tails, while requiring substantially weaker moment assumptions. We further extend the Gaussian approximation to VC-subgraph classes and apply it to robust vector mean estimation under arbitrary norms, obtaining bounds with optimal Gaussian-width complexity. Finally, we develop uniform confidence intervals based on the bootstrap approximation and show empirically that they maintain coverage under heavy tails and adversarial contamination.
\end{abstract}


\section{Introduction}

% - falar de robustez em termos de estimação de media no sentido de
% |muhat - mu| < sigma raiz(1/n ln 1/a)

% - outro fator comum na literatura moderna de robustez é a presença de contaminação adversarial.

% - existem varios estimadores que alcançam isso: median of means, trimmed means.

% - outra propriedade interessante e util da media amostral é a aproximação gaussiana, além de approximações via bootstrap que permitem produzir intervalos de confiança, construir testes de hipóteses, dentre outras aplicações.

% - Em termos de aplicações, o regime de alta dimensão (n >> d) é especialmente relevante em aplicações atuais e tem tido espaço na literatura recente já que resultados classicos de berry-esseen não são suficientes.

% - nesse artigo buscamos entender as propriedades de aproximação gaussiana em alta dimensão de estimadores robustos olhando para a trimmed mean.

% - definir e dar contexto da trimmed mean.

% - Seguindo chernozhukov, vamos olhar para a estatística Z e entender o comportamento não-assintótico de P(Z<l) - P(N<l).

Let $X, X_1, X_2, \dots, X_n$ be i.i.d. random variables taking values in some set $\bX$ according to a distribution $P$. Also let $\sF$ be a family of functions $f: \bX\to\R$ and denote by $Pf$ the expectation of $f(X)$. Uniformly estimating the population mean $Pf$ for all $f \in \sF$ is a fundamental task in statistics; when $\sF$ has a single element it is the classical problem of mean estimation, when $\sF$ is finite it is related to the problem of vector mean estimation under the $\| \cdot \|_\infty$-norm, when $\sF$ is infinite this problem is related to vector mean estimation under general norms, maximum likelihood estimation, generalization bounds, and others.

Since Catoni's seminal paper \cite{Catoni2012}, a line of research has focused on designing mean estimators attaining, with few assumptions on $P$ and sometimes also considering contaminated samples, high probability error bounds similar to those the empirical mean attains on sub-Gaussian samples. Such estimators are called sub-Gaussian mean estimators, see \cite{Devroye2016, minsker2018uniform, Lugosi2019a, Lugosi2019, Lugosi2021}. For instance, taking $\sF = \{f\}$ and assuming only $\sigma^2 := \Ex (f(X)-Pf)^2 < \infty$, it is possible to construct a family of estimators $\left\{\muhat_{n,\alpha} : \R^n \to \R \right\}_{\alpha \in (0,1)}$ satisfying, for some absolute constant $C>0$,
\begin{equation}
    \label{eq:highprobability}
    \Pr\left\{ \left| \muhat_{n,\alpha}(X_1, \dots, X_n) - Pf \right| \leq C \sqrt{\frac{\sigma^2}{n}\ln \frac{1}{\alpha}}\right\} \geq 1 - \alpha.
\end{equation}

% Analogous results exist in the literature assuming contaminated data and general families $\sF$ \cite{Lugosi2021, OliveiraResende}. The underlying goal of robust mean estimation is to achieve the same concentration bounds the empirical mean achieves on sub-Gaussian samples, but with minimal assumptions over $P$.
This article looks at robustness in terms of another desirable property of the empirical mean under light-tailed distributions: its high-dimensional Gaussian and bootstrap approximations. We study high-dimensional Gaussian and bootstrap approximations for the trimmed mean, a mean estimator that is known to achieve optimal performance in several problems \cite{OliveiraResende}. Let $k < \frac{n}{2}$ be a non-negative integer, the trimmed mean estimate of $Pf$ given $x_{1:n} = (x_1, \dots, x_n) \in \bX^n$ is simply
\[ T_{n,k}(f, x_{1:n}) = \frac{1}{n-2k} \sum_{i=k+1}^{n-k} f\left(x_{(i)}\right) \]
where $(\cdot)$ is a permutation such that $f\left(x_{(1)}\right)\leq f\left(x_{(2)}\right)\leq \cdots \leq f\left(x_{(n)}\right)$. Following the recent work on high-dimensional Gaussian approximation for the empirical mean  \cite{chernozhukov2013gaussian, chernozhukov2016empirical, koike2021notes, chernozhuokov2022improved, giessing2023gaussian}, we let $\{G_P(f)\}_{f \in \sF}$ be a centered Gaussian process with the same covariance as $\left\{f(X)\right\}_{f \in \sF}$ and look at the Kolmogorov distance between the maximum of the Gaussian process and the maximum of the finite-sample approximation. We show that the trimmed mean satisfies the Gaussian approximation property  
\begin{equation}
\label{eq:main_contribution}
    \sup_{\lambda \in \R} \left| \Pr\left[ \sup_{f \in \sF} \sqrt{n}\, \left(T_{n,k}(f, X_{1:n}) - Pf\right) \leq \lambda \right] -  \Pr\left[ \sup_{f \in \sF} G_Pf \leq \lambda \right] \right| \to 0
\end{equation}
for contaminated heavy-tailed data even in the high-dimensional setting where $|\sF| \gg n$. This line of research differs from classical Berry-Esseen results \cite{nagaev1976estimate,bentkus2003dependence, bentkus2005lyapunov, prokhorov2000limit, klartag2012variations} since our main concern is allowing for $|\sF| \gg n$, usually paying the price of having a slower vanishing rate and obtaining an approximation bound only over the maximum and not over all convex sets, as in the low-dimensional case.

As discussed in \cite{chernozhukov2023high, belloni2018high}, Gaussian and bootstrap approximations for the maximum appear in applications such as constructing simultaneous confidence sets, multiple hypothesis testing, and treatment heterogeneity analysis. High-dimensional settings frequently arise in fields like biology and econometrics, mainly due to sample size limitations or the presence of parameters controlling for heterogeneity. However, as shown by \cite{kock2024remark}, the empirical average fails to attain Gaussian approximation properties for heavy-tailed distributions when the dimension is large (see \S\ref{subsec:limitations} for details). 

Motivated by this limitation, the main contribution of this paper is establishing \eqref{eq:main_contribution} for contaminated heavy-tailed data in the regime $|\sF| \gg n$. We explore this result through two applications:
\begin{itemize}
    \item \textbf{Vector mean estimation under arbitrary norms:} We extend our Gaussian approximation results from finite-dimensional settings to VC-subgraph classes. Using this extension in place of classical empirical process tools, we achieve optimal complexity bounds with a Gaussian width term.
    \item \textbf{Uniform confidence intervals:} Building on our bootstrap approximation results, we propose a method for constructing uniform confidence intervals. We provide experimental evidence that these intervals, when obtained via trimmed means, remain valid under heavy tails and adversarial contamination.
\end{itemize}

This paper is organized as follows. In \S\ref{subsec:assumptions} we introduce the main definitions and notations. In \S\ref{subsec:literature} we provide a short literature overview on high-dimensional Gaussian and bootstrap approximation. Our Gaussian approximation results are presented in \S\ref{sec:mainresult}. In \S\ref{sec:vectormean} we apply our results to the problem of vector mean estimation under general norms. We provide bootstrap approximations and their application to uniform confidence intervals in \S\ref{sec:bootstrap}. The proofs of all results from \S\ref{sec:mainresult} are presented in \S\ref{sec:proofideas}. All other proofs are in the supplementary material. We observe that our proof techniques can be easily combined with new, improved results for Gaussian approximation of the empirical average, potentially resulting in tighter bounds than those presented here.

\begin{remark}[Concurrent and subsequent work \cite{liu2024robust, kock2025high}] 
This article is based on work originally presented in the author's PhD thesis \cite{phd}. A few months after the thesis was submitted, \cite{liu2024robust} posted an independent preprint exploring Gaussian and bootstrap approximations for a robust estimator that combines the median-of-means with a winsorization procedure. Furthermore, during the review process of this manuscript, \cite{kock2025high} released a preprint that improved upon some of the bounds established in our initial draft. Prompted by their findings, a careful review of the original manuscript's proofs revealed a suboptimal bounding step. By resolving this lack of sharpness, the bounds in the present article have been improved to match those derived in \cite{kock2025high}. The new solution also motivated an improvement of Theorem \ref{thm:ga_vc} and on Theorem \ref{thm:bootstrap_approximation}. The author is grateful to M. E. Lopes for bringing his concurrent work to the author's attention, and to A. B. Kock for stimulating discussions, particularly regarding his interest in Lemma \ref{lemma:epsbound}, which was added subsequent to the original version of this paper.\end{remark}

\subsection{Definitions and notations}\label{subsec:assumptions}

% \Lucas{VERIFICAR CENTRALIZACAO. VERIFICAR TAMBEM CONSTANTES OU HIPOTESES USADAS SEM CLARIDADE.}

Recall that $\bX$ is a set, $\sF$ is a family of functions $f : \bX \to \R$, $P$ is a distribution over $\bX$ and $X, X_1, \dots, X_n$ are i.i.d. random variables with law $P$. We denote $\Ex f(X)$ by $Pf$, use $|\cdot|$ to denote the cardinality of a set and let $[n] := \{ 1, 2, \dots, n \}$ for every positive integer $n$. Given a family $\sF$ of square-integrable functions we define
\[ \Sigma_{\sF, P}(f,g) := \mbox{Cov}(f,g), \,\, f,g \in \sF, \]
which is simply the covariance matrix of $\sF$ when $|\sF| < \infty$ or the kernel of $\sF$ when $|\sF| = \infty$. We let $\{G_P(f): f \in \sF\}$ be a centered Gaussian process such that
\[ \Ex \left[G_P(f)G_P(g)\right] = \Sigma_{\sF, P}(f,g)\,\, \forall f,g \in \sF. \]
See \cite{Vaart1996} for details on the existence of such processes.

\begin{definition}[Adversarially contaminated sample]\label{def:contamination} Let $\varepsilon \in [0,1]$, $X_{1:n}^\varepsilon = (X_1^\varepsilon, \dots, X_n^\varepsilon)$ is an $\varepsilon$-contaminated version of $X_{1:n}$ if
\[ |\{ i \in [n] : X_i^\varepsilon \neq X_i \}| \leq \varepsilon n. \]
In this setup, one has access only to the contaminated sample $X_{1:n}^\varepsilon$, which differs from the clean sample in at most $\eps n$ points. This contamination setup is said to be adversarial because no assumption is made on how the contamination procedure works. It has been widely used on the recent literature (see \textit{e.g.} \cite{Depersin2021, Lecue2020, Diakonikolas2019, diakonikolas2019a, lecue2019learning}) and is more adverse than Huber's classical contamination model \cite{huber1965robust, huber1981robust}, where the distribution $P$ is assumed to be known and the contaminated points are drawn from an unknown distribution. This setup is closely related to the classical notion of breakdown point \cite{huber1981robust}: an estimator that is robust to an $\eps$-contamination has a breakdown point of at least $\eps$. For further discussion and comparisons with other contamination procedures see \cite{diakonikolas2023algorithmic}.
\end{definition}

This model is realistic in a range of modern data-analysis settings where a small fraction of observations may be corrupted by mechanisms that are difficult to model parametrically. In high-throughput biology, sample mislabeling and contamination are well documented \cite{zych2017regenotyper}. Similar sparse-corruption settings arise in sensor networks, where a fraction of sensors may be faulty or compromised \cite{fagnani2014distributed}, and in crowdsourced or web-scale datasets, where a few percent of labels may be erroneous \cite{northcutt2021pervasive}. In these settings, the contaminating mechanism is unknown and need not follow a specified outlier distribution, motivating an adversarial rather than a Huber contamination model. The requirement in our results that $\epsilon$ vanish with $n$ (see \S\ref{sec:mainresult}) reflects the fact that distributional (Gaussian) approximation is a strictly harder goal than point estimation, and Lemma \ref{lemma:epsbound} shows that this requirement is unavoidable.

Given $p > 2$ we let
\[ \nu_p(\sF, P) := \sup_{f \in \sF} \left(P|f-Pf|^p\right)^\frac{1}{p} \text{ and } \underline{\sigma}^2_{\sF, P} := \inf_{f \in \sF}\Sigma_{\sF, P}(f,f). \]
We may drop the dependence on $\sF$ and $P$ when it is implicit and simply write $\Sigma$, $\nu_p$, and $\underline{\sigma}^2$.

Finally, we let $\Tmhat_{n,k}^\eps(f) := T_{n,k}\left(f, X_{1:n}^\eps\right)$ and let $\Pmhat_n(f)$ denote the empirical mean of $f$. We define the normalized empirical processes corresponding to the empirical mean and the trimmed mean as
\[ \G_n(f) := \sqrt{n} \left( \Pmhat_n(f) -Pf \right) \quad \text{ and } \quad \T_{n,k}^\eps(f) := \sqrt{n}\left(\Tmhat_{n,k}^\eps(f) - Pf\right), \]
respectively. Letting the suprema of these processes over the class $\sF$ be given by
\[ Z_n(\sF) := \sup_{f \in \sF}\, \G_n(f), \quad Z_{n,k}^\eps(\sF) := \sup_{f \in \sF}\, \T_{n,k}^\eps(f), \quad \text{ and } \quad Z(\sF) := \sup_{f \in \sF}\, G_P(f), \]
we define the Kolmogorov-Smirnov distances between the distributions of these suprema and $G_P$ as
\[
    \varrho^E := \sup_{\lambda \in \R } \left| \Pr\left( Z_n(\sF) \leq \lambda \right) - \Pr\left( Z(\sF) \leq \lambda \right) \right| \quad\text{ and }\quad
    \varrho := \sup_{\lambda \in \R } \left| \Pr\left( Z_{n,k}^\eps(\sF) \leq \lambda \right) - \Pr\left( Z(\sF) \leq \lambda \right) \right|.
\]

% \begin{remark}
% We highlight that while our results are presented assuming an i.i.d. sample $X_{1:n}$, the same considerations hold for a sample $X_{1:n}$ of independent random variables satisfying $\Ex f(X_i) = \Ex f(X_j)$ for all $f\in\sF$ and $i,j\in [n]$. In that case the covariance must be redefined as
% \[ \Sigma = \frac{1}{n}\sum_{i=1}^n \Ex\left( (f(X_i)-\Ex f(X_i)) (f(X_j)-\Ex f(X_j)) \right) \]
% and $\nu_p$ must be taken as the worse case over the distributions of all $X_i$. We keep the presentation assuming independence for clarity of notation.
% \end{remark}

\subsection{Literature overview}
\label{subsec:literature}

\subsubsection{Gaussian and bootstrap approximations for the empirical mean}

Under strong moment assumptions, the pioneering work of \cite{chernozhukov2013gaussian} first bounded $\varrho^E \leq C \left( \frac{\ln^7(nd)}{n} \right)^\frac{1}{6}$.
This bound was then improved by \cite{koike2021notes} and \cite{chernozhuokov2022improved} to 
$\varrho^E \leq C \left( \frac{\ln^5(nd)}{n} \right)^\frac{1}{4}$.
Under additional assumptions, such as $\Sigma_{\sF, P}$ being positive definite, $\varrho^E$ can be further bounded, as discussed in Remark \ref{remark:strongcov}. We say that these are high-dimensional Gaussian approximation bounds because they ensure $\varrho^E \to 0$ even when the dimension $d$ is exponential on a polynomial of $n$. This behavior contrasts with Berry-Esseen-type bounds \cite{nagaev1976estimate,bentkus2003dependence, bentkus2005lyapunov, prokhorov2000limit, klartag2012variations} since they only allow for a polynomial dependence between $d$ and $n$. On the other hand, Berry-Esseen-type bounds provide an uniform bound over all convex sets and not only for the maximum. There is also literature on high-dimensional Gaussian approximation for the empirical average in the case $d=\infty$. Assuming the existence of an envelope function and some moment conditions, a line of work started by \cite{Chernozhukov2014, chernozhukov2016empirical} derived bounds depending on the metric entropy of the class $\sF$. More recently, dimension and entropy free bounds were obtained by \cite{giessing2023gaussian}, although also requiring an envelope function with finite third moment.

All the previously cited works also consider bootstrap approximation results. The most popular bootstrap strategies are the empirical and the multiplier bootstrap. The usual approach is to prove the validity of bootstrap by showing that both the mean and the bootstrap statistics satisfy a Gaussian approximation result with the same limit. Some work has also obtained bootstrap approximations directly: \cite{deng2020beyond} managed to improve upon the then-current state of the art using a direct bootstrap approximation. Moreover, \cite{lopes2020bootstrapping} obtained a bootstrap approximation rate of $n^{-\frac{1}{2}}$ independent of the dimension, but under sub-exponential tails and variance decay.

\subsubsection{The limitations of the empirical mean}\label{subsec:limitations}

All the previously cited works on high-dimensional Gaussian and bootstrap approximations require strong moment assumptions to ensure $\varrho^E \to 0$ when $d$ is exponential on a polynomial of $n$. Considering the case $|\sF|<\infty$, the current state-of-the-art Gaussian approximation bound \cite[Theorems 2.1 and 2.5]{chernozhuokov2022improved} requires the existence of $B$ such that $\nu_4^4 \leq B^2 \nu_2^2$ for all $f \in \sF$, $\underline{\sigma}^2 > 0$ and at least one of the following two conditions: (i) $\Ex e^\frac{|f(X)-Pf|}{B} \leq 2$ for all $f \in \sF$ or (ii) $\Ex \max_f |f(X)-Pf|^p \leq B^p$ for some $p>2$. Moreover, under condition (ii) they obtain
\begin{equation}
    \label{eq:sotaempirical}
    \varrho^E \leq C\left( B\frac{(\ln d)^{\frac{3}{2} - \frac{1}{p}}}{n^{\frac{1}{2} - \frac{1}{p}}} + \left( 
\frac{B^2 \ln^5d}{n} \right)^\frac{1}{4} \right)
\end{equation}
for some constant $C$ depending only on $\underline{\sigma}^2$ and $\nu_2$. Similarly, for $|\sF|=\infty$, \cite[Theorem 2.1]{chernozhukov2016empirical} requires the existence of a envelope function $F : \bX \to \R$ satisfying $F(x) \geq |f(x)| ~ \forall x \in \bX, f\in \sF$ and $PF^4 < \infty$. Meanwhile, \cite[Theorem 6]{giessing2023gaussian} requires $PF^3 < \infty$ and provides a bound that depends on $PF^3$. The following theorem justifies the need of such strong moment assumptions when studying Gaussian approximations for the empirical mean:

\begin{theorem}[Threshold phenomenon for the Gaussian approximation for the empirical average; adapted from Theorems 2.1 and 2.2 of \cite{kock2024remark}] \label{thm:threshold} Let $\sF$ be the family of all $d$ coordinate projections of $\R^d$ and let $p \in (2,\infty)$. 
\begin{itemize}
    \item There exists a centered distribution $P$ with i.i.d. marginals, finite $p$-th moment, but no higher moments, and such that if $\delta > 0$ satisfies $\limsup_{n\to\infty} dn^{1-\frac{p}{2} - \delta} > 0$ and $X_{1:n} \sim P^n$, then
\[ \limsup_{n\to\infty} \sup_{\lambda \in \R} \left| \Pr\left(Z_n(\sF) \leq \lambda\right) - \Pr\left(Z(\sF) \leq \lambda\right) \right| = 1. \]
    \item Moreover, given $0 < c \leq C^\frac{2}{p} < \infty$, let $\mathcal{P}(c, C, p)$ be the class of all centered distributions $P$ over $\R^d$ with $\underline{\sigma}_{\sF, P} \geq c$ and $\nu_p^p(\sF, P) \leq C$. If $\delta > 0$ satisfies $\lim_{n\to\infty} dn^{1-\frac{p}{2} + \delta} = 0$, then
    \[ \lim_{n\to\infty} \sup_{\lambda \in \R} \sup_{P \in \mathcal{P}(c,C,p)} \left| \Pr\left(Z_n(\sF) \leq \lambda\right) - \Pr\left(Z(\sF) \leq \lambda\right) \right| = 0 \]
\end{itemize}
\end{theorem}

Thus, the Gaussian approximation for the empirical average has a phase transition at $d = n^{\frac{p}{2}-1}$, being feasible for a large class of distributions when $d \ll n^{\frac{p}{2}-1-\delta}$, but not when $d \gg n^{\frac{p}{2}-1+\delta}$. Notice that \eqref{eq:sotaempirical} matches Theorem \ref{thm:threshold} since, in general, the inequality  $\Ex \max_f |f(X)|^p \leq d \nu_p^p$ cannot be improved\footnote{For instance, take functions $f \in \sF$ with disjoint support and the same $p$-th moment.} and so $B$ must be of order $d^\frac{1}{p}\nu_p$, yielding $\varrho^E \to 0$ when $d \ll n^{\frac{p}{2}-1}$.

\subsubsection{Gaussian and bootstrap approximations for robust estimators}

Gaussian and bootstrap approximations for robust estimators were studied by \cite{liu2024robust} and \cite{kock2025high} for the finite-dimensional (\textit{i.e.}, $|\sF| < \infty$). In \cite{liu2024robust}, assuming the uncontaminated setting ($\eps = 0$), the authors combine a median-of-means approach with a variance-based Winsorization procedure to establish high-dimensional Gaussian and bootstrap approximation guarantees. Their estimator uses a small subset of the sample to construct robust mean and variance estimates, $\widehat{\mu}_f$ and $\widehat{\sigma}_f^2$, for all $f \in \sF$. These estimates are then used to construct a Winsorized mean, normalized by $\sqrt{n}\widehat{\sigma}_f^\tau$ for a tuning parameter $\tau \in (0,1)$.

Under an $L^{4+\delta}-L^2$ moment equivalence condition, bounded densities, and a variance decay assumption, the authors achieve a nearly optimal Gaussian approximation rate of $n^{-1/2 + \delta}$. Their result relies heavily on the variance decay condition, which requires the existence of constants $C \ge 1$ and $\beta > 0$ such that
\[
    \frac{1}{C}\sigma_{(1)}^2 j^{-2\beta} \leq \sigma_{(j)}^2 \leq C\sigma_{(1)}^2 j^{-2\beta}
\]
for the sorted variances $\sigma_{(1)}^2 \ge \dots \ge \sigma_{(d)}^2$ associated with the coordinates $f_1, \dots, f_d \in \sF$. This variance decay profile is crucial for proving that there exists, with high probability, a low-dimensional subset $\sH \subset \sF$ that contains the index realizing the maximum statistic. Notably, the choice of $\tau \in (0,1)$ is of particular importance to their results; if $\tau$ were set to $1$, the normalization would place all coordinates on an equal scale, rendering the subset $\sH$ impossible to construct.

Building on the results from the original manuscript of this paper, \cite{kock2025high} addressed Gaussian and bootstrap approximations for both a quantile-based Winsorized and the trimmed means, which they show to satisfy Gaussian approximation properties as long as $\ln |\sF| \ll n^{\frac{1}{5} + O\left( p^{-1} \right)}$. They do not cover the case $|\sF| = \infty$ nor the problem of vector mean estimation under general norms. Their proof strategy consists in approximating the Winsorized mean given by empirical quantiles by the one given by populational quantiles, controlling the error, and using the Gaussian approximation result from \cite{chernozhuokov2022improved}. Meanwhile, inspired by \cite{OliveiraResende}, we approximate $\widehat{T}_{n,k}^\eps(f)$ by $\Pmhat_n( \tau_M\circ f )$, where $\tau_M(x) = x\mathbf{1}_{|x| \leq M} + M\mathbf{1}_{x > M}  - M\mathbf{1}_{x < -M}$ and control $M$, without requiring a fine control over the quantiles. 
% More recently, the same authors \cite{kock2026robustness}  explored these approximations to propose robust max-tests for high-dimensional means.

\section{High-dimensional Gaussian approximations}\label{sec:mainresult}

Recall the definitions from \S\ref{subsec:assumptions}. In this section we let $|\sF| = d < \infty$. Our main result follows:

\begin{theorem}[High-dimensional Gaussian approximation for trimmed means, proof in \S\ref{subsec:proof_gaussian}]\label{thm:gaussian_approximation} Assume an $\eps$-contaminated sample $X_{i:n}^\eps$ as in \S\ref{subsec:assumptions}. Let $nd \geq 3$ and $\nu_p := \nu_p(\sF, P) < \infty$ for some $p \in (2,\infty)$. If $k := \left\lfloor \eps n \right\rfloor + \left( \left\lfloor \eps n \right\rfloor \vee \left\lceil \ln(nd) \right\rceil \right) < \frac{n}{2}$ and the regularity condition
\begin{equation}\label{eq:condition_b1}
    \nu_p^2 \left(\frac{k}{n}\right)^{1-\frac{2}{p}} \leq \frac{\underline{\sigma}^2_{\sF, P}}{8}
\end{equation}
holds, then
\[\varrho \leq c_1 \sqrt{\nu_p} \left( \frac{\ln^{5-\frac{2}{p}}(nd)}{n^{1-\frac{2}{p}}} \right)^\frac{1}{4} + c_2 \nu_p \left( \eps^{1-\frac{1}{p}}\sqrt{n \ln(nd) } +  \left( \frac{\ln^{3-\frac{2}{p}}(nd)}{n^{1-\frac{2}{p}}} \right)^\frac{1}{2} \right) + \frac{1}{n} \]
where $c_1 > 0$ depends on $\nu_2(\sF, P)$ and $\underline{\sigma}_{\sF, P}$, and $c_2 \geq \underline{\sigma}_{\sF, P}^{-1}$ depends only on $\underline{\sigma}_{\sF, P}$.
\end{theorem}

Theorem \ref{thm:gaussian_approximation} can be compared with the best-known Gaussian approximation bound available for the empirical mean, given by \eqref{eq:sotaempirical}. For a baseline comparison, we first set $\eps = 0$, as the standard empirical mean cannot handle adversarial contamination. To ensure $\varrho^E \to 0$ in \eqref{eq:sotaempirical}, one requires light-tailed distributions and a dimension growth restricted to $\ln(d) \ll n^\frac{1}{5}$. In contrast, to obtain $\varrho \to 0$ in Theorem \ref{thm:gaussian_approximation}, it suffices to have $\nu_p < \infty$ for some $p>2$ and $\ln(d) \ll n^{\frac{1}{5} - O\left( p^{-1} \right)}$ (indeed, the binding term forces $\ln d \ll n^{(1-2/p)/(5-2/p)}$, whose exponent equals $\frac{1}{5} - \frac{8}{25p} + o(p^{-1})$ and increases to $\frac{1}{5}$ as $p\to\infty$). While our allowable dimension growth is slightly more restrictive than the $n^{1/5}$ rate available for the empirical mean, our result holds under significantly weaker moment assumptions and remains valid under data contamination.

Furthermore, Theorem \ref{thm:gaussian_approximation} operates within a vanishing contamination framework. The bound on $\varrho$ vanishes only if the contamination level $\eps$ satisfies
\begin{equation}
\label{eq:eps_bound}
\frac{\nu_p}{\underline{\sigma}_{\sF, P}} \eps^{1-\frac{1}{p}} \sqrt{n\ln(nd)} \to 0.
\end{equation}
This imposes a requirement that $\eps$ must decay with both the sample size $n$ and the dimension $d$, which contrasts with the fixed breakdown point restrictions (e.g., $\eps < c$ for some $c>0$) typically found in robust sub-Gaussian mean estimation. Lemma \ref{lemma:epsbound} shows that this dependence is not incidental, but actually sheds light on a fundamental limitation of Gaussian approximations under contamination.

\begin{lemma}[Upper bound on $\eps$; proof in \S\ref{sm:proof_epsbound}]\label{lemma:epsbound} Assume $p>2$. Let $b_p \geq \sigma > 0$. Let $\delta > 0$, there exist constants $o_1, d_1, c_1 > 0$ depending on $\delta$ and a constant $c_2$ depending on $\delta$ and $\sigma$ such that if $d > d_1$, $\eps n \geq o_1$,
\begin{equation}
    \label{eq:opt_eps_bound}
    \frac{b_p}{\sigma}\eps^{\frac{1}{2}-\frac{1}{p}} \leq 1 \wedge \frac{c_2}{\ln d} \quad\text{ and }\quad
    \eps^{1-\frac{1}{p}} \geq \frac{2\delta}{b_p} (2 n \ln d)^{-\frac{1}{2}},
\end{equation}
then, one can construct a family $\sF$ with $|\sF| = d$ and distributions $P, P'$ satisfying $b^p_p = \nu^p_p(\sF, P) \leq \nu^p_p(\sF, P') \leq 2 b^p_p$ and $\sigma^2 = \underline{\sigma}_{\sF, P}^2 = \nu^2_2(\sF, P) \leq \nu^2_2(\sF, P') = \underline{\sigma}_{\sF, P'}^2 \leq 2\sigma^2$, such that
for any given estimator $\widehat{E}_f : \bX^n \to \R$ it holds
\[ \sup_{\lambda \in \R } \left| \Pr\left( \sup_{f \in \sF} \sqrt{n} \left(\hat{E}_f\left(X_{1:n}^\eps\right)-Pf\right) \leq \lambda \right) - \Pr\left( \sup_{f \in \sF} G_P(f) \leq \lambda \right) \right| \geq c_1 \]
or
\[ \sup_{\lambda \in \R } \left| \Pr\left( \sup_{f \in \sF} \sqrt{n} \left(\hat{E}_f\left(X_{1:n}^{'\eps}\right)-P'f\right) \leq \lambda \right) - \Pr\left( \sup_{f \in \sF} G_{P'}(f) \leq \lambda \right) \right| \geq c_1. \]
\end{lemma}

Observe that the inequalities in \eqref{eq:opt_eps_bound} can simultaneously hold only if $d$ is large enough and
\[ n \geq C(\delta, \sigma, b_p) (\ln d)^{4\frac{p-2}{p-1} - 1} \]
for a given constant $C(\delta, \sigma, b_p)$. In particular, as $p\to\infty$, this imposes a sample size requirement of at least $\Omega(\ln^3 d)$. This regime is non-restrictive, as $n \gg \ln^3 d$ is a necessary condition for high-dimensional Gaussian approximations for bounded random variables in the absence of contamination \cite[Proposition 2.1]{chernozhukov2023nearly}. Consequently, Lemma \ref{lemma:epsbound} establishes that a valid Gaussian approximation under adversarial contamination is fundamentally possible only if the contamination level vanishes fast enough to satisfy:
\begin{equation}
\label{eq:opt_eps_limit}
\frac{b_p}{\sigma} \eps^{1-\frac{1}{p}} \sqrt{n \ln d} \to 0.
\end{equation}
Since the dependence on the contamination level from Theorem \ref{thm:gaussian_approximation} matches the bound of Lemma \ref{lemma:epsbound}, it is optimal. The limit established in Lemma \ref{lemma:epsbound} must be compared with two results from the literature: Lemma 5.4 of \cite{minsker2018uniform} and Theorem 10 of \cite{deng2020beyond}. First, \cite{minsker2018uniform} demonstrates that the optimal error rate for robust mean estimation under heavy tails and adversarial contamination scales as $b_p \eps^{1-1/p}$. Second, \cite{deng2020beyond} shows that the $\frac{\sqrt{\ln d}}{\sigma}$ penalty arising from Nazarov's inequality is sharp in certain cases. Because establishing a Gaussian approximation requires standardizing the error by a scaling factor of $\sqrt{n}$, it yields a cumulative error term proportional to $\nu_p \eps^{1-1/p} \sqrt{n \ln d}$, matching \eqref{eq:opt_eps_limit}.

\begin{remark}[Role of the covariance matrix \label{remark:strongcov}]
Theorem \ref{thm:gaussian_approximation} and all results in this paper require only that $\underline{\sigma}_{\sF, P} > 0$, without strictly demanding a positive definite $\Sigma_{\sF, P}$. This mirrors the setup used by \cite{chernozhuokov2022improved} to derive \eqref{eq:sotaempirical}. By contrast, imposing stronger structural assumptions is known to yield tighter bounds for the empirical mean under light tails:
\begin{itemize}
    \item \textbf{Positive definiteness:} Assuming a positive definite $\Sigma_{\sF, P}$ alongside bounded coordinates for $X_i$, \cite{chernozhukov2023nearly} established a bound of order $(\ln n)n^{-1/2} \ln^{3/2}(d)$, achieving optimality up to the logarithmic factor. Similar improvements were noted by \cite{kuchibhotla2020high}.
    \item \textbf{Symmetry and variance decay:} Symmetric distributions \cite{chernozhuokov2022improved} and variance decay \cite{lopes2020bootstrapping} also allow the empirical mean to attain the optimal $n^{-1/2}$ rate under light tails.
\end{itemize}

While extending in these directions is beyond the scope of this article, the proof techniques developed in \S\ref{sec:proofideas} provide a natural pathway to adapt these conditions for the trimmed mean, potentially yielding analogous improvements.\end{remark}

\subsection{Gaussian approximation for VC classes}\label{sec:empirical_vector}

Before discussing our results for empirical processes we need a few definitions. Given a measure $Q$ and a family $\sF$ of square-integrable we let $\sN(\sF, d_Q, \delta)$ be the $\delta$-covering number of $\sF$ with respect to the semi-metric $d_Q(f,g) = \sqrt{Q(f-g)^2}$.

\begin{definition}[VC-subgraph class] We say that a class $\sF$ of functions $f: \bX \to \R$ is a VC subgraph class with dimension $v = \text{vc}(\sF)$ if $v< \infty$ is the VC dimension of the collection of its subgraphs, i.e.,
\[ v := \text{vc}\left(\left\{ \left\{(x,t) \in \bX \times \R: t<f(x) \right\} : f \in \sF \right\}\right) < \infty. \]
\end{definition}

\begin{definition}[VC-type class]\label{def:vctype} We say that a class $\sF$ of functions $f: \bX \to \R$ is a VC-type class with envelope $F$ if there are constants $A,v > 0$ such that
\[ \sup_{Q} \sN\left( \sF, d_Q,  \delta \| F\|_{L^2(Q)} \right) \leq \left( \frac{A}{\delta} \right)^{v}  \,\, \forall \delta \in (0,1].\]
Where the supremum is taken over all probability measures over $\bX$ with finite support.
\end{definition}

Our result on the Gaussian approximation for empirical processes follows:

\begin{theorem}[Gaussian approximation for VC classes, proof in \S\ref{subsec:proof_ga_vc}]\label{thm:ga_vc}
    Assume an $\eps$-contaminated sample $X_{i:n}^\eps$ as in \S\ref{subsec:assumptions}. Let $\sF$ be a VC-subgraph class and let $p\in (2,\infty)$ be such that $\nu_p := \nu_p(\sF, P)< \infty$. Define
    \[ K_n := K_n(\sF) = \lceil 12(\text{vc}(\sF) + 1) \ln n \rceil \quad \text{and } \quad \Xi(\delta) = \Ex\left[ \sup_{\substack{f,g \in \sF \\ d_P(f,g) < \delta}} G_P(f-g) \right]. \]
    Suppose that \eqref{eq:condition_b1} holds taking $k := \lfloor \eps n\rfloor + \left( \lfloor \eps n \rfloor \vee K_n \right) < \frac{n}{2}$. Then,
    \begin{align*}
        \varrho \leq c_1 \sqrt{\nu_p} \left( \frac{K_n^{5-\frac{2}{p}}}{n^{1-\frac{2}{p}}} \right)^\frac{1}{4} + c_2 \nu_p \left( \eps^{1-\frac{1}{p}}\sqrt{n K_n } +  \left( \frac{K_n^{3-\frac{2}{p}}}{n^{1-\frac{2}{p}}} \right)^\frac{1}{2} \right) + \frac{5\sqrt{K_n}}{2\underline{\sigma}_{\sF, P}}\,\Xi\left(14 \nu_p \left( \frac{K_n}{n} \vee \eps \right)^{\frac{1}{2}-\frac{1}{p}}\right) + \frac{2}{n}
    \end{align*}
    where $c_1 > 0$ depends on $\nu_2(\sF, P)$ and $\underline{\sigma}_{\sF, P}$, and $c_2 \geq \underline{\sigma}_{\sF, P}^{-1}$ depends only on $\underline{\sigma}_{\sF, P}$.
\end{theorem}

\begin{remark}\label{rem:boundXi}
    The quantity depending on $\Xi$ is a continuity modulus of the limiting Gaussian process and vanishes as the net becomes finer. First notice that if $G_P$ has uniformly $d_P$-continuous sample paths, then $\Xi(\delta) \to 0$ as $\delta \to 0$, so this term is asymptotically negligible. More precisely, non-asymptotic bounds can be derived using entropy integral bounds. For instance, if $\sF$ is a VC-type class (as in Definition \ref{def:vctype}) one can easily see that, for $\delta < \frac{A\|F\|_{L^2(P)}}{e}$,
    \[ \Xi(\delta) \leq C \int_0^\delta \sqrt{\text{vc}(\sF) \ln \frac{A\|F\|_{L^2(P)}}{s}}\, ds \leq C' \delta \sqrt{\text{vc}(\sF) \ln \frac{A\|F\|_{L^2(P)}}{\delta}} \]
    for some absolute constants $C, C'>0$. We stress that, plugging this bound with $\delta \asymp \nu_p (K_n/n \vee \eps)^{\frac12 - \frac1p}$, the resulting $\Xi$-term can be absorbed into the other terms (see the proof of Theorem \ref{thm:generalnorm}).
\end{remark}

Theorem \ref{thm:ga_vc} establishes the first Gaussian approximation result for a robust estimator in the infinite-dimensional setting. In the uncontaminated setting ($\eps = 0$), Theorem \ref{thm:ga_vc} can be compared with the seminal results of \cite[Theorem 2.1]{chernozhukov2016empirical}. The convergence rate established by \cite{chernozhukov2016empirical} is of order at least $( K_n^7/n )^{1/8}$ (see \S\ref{sm:order_ga_vc} for a discussion), whereas our rate improves to $( K_n^5/n )^{1/4}$ in the limit as $p\to\infty$. Furthermore, unlike \cite{chernozhukov2016empirical}, Theorem \ref{thm:ga_vc} circumvents the need for an envelope, requiring only $\nu_p(\sF, P)<\infty$ for some $p \in (2, \infty)$ rather than an envelope in $L^p(P)$ for $p \ge 4$. Even if one were to update the analysis of \cite{chernozhukov2016empirical} using their more recent finite-dimensional approximation bounds (e.g., Lemma \ref{thm:ck_gaussian_approximation}), which could theoretically improve their polynomial rate to $1/4$, the requirement for an envelope $F \in L^p(P)$ would still remain. Theorem \ref{thm:ga_vc} also fares well when compared against \cite{giessing2023gaussian}; while their bounds do not require a VC-subgraph class, they yield a worse dependence on $n$ and demand an envelope in $L^3(P)$.

\section{Vector mean estimation under general norms}\label{sec:vectormean}

In this section we apply our Gaussian approximation results to the following problem:

\begin{problem}[Robust vector mean estimation under a general norm]\label{problem:vectormean} Let $\bX = \R^d$, $\Pm$ be a distribution over $\bX$ and $\| \cdot \|$ be a norm in $\R^d$ satisfying $P\|\cdot\| < \infty$. The mean $\mu_P$ of $P$ is characterized by $\langle v, \mu_P \rangle = P\langle v, \cdot \rangle$ for all $v \in S$, where $S\subset \R^d$ is symmetric and satisfies
\[\forall x\in \R^d\,:\,\|x\| = \sup_{v\in S}\langle x,v\rangle.\]
Let $X_{i:n}^\eps$ be an $\eps$-contaminated sample and $\alpha \in (0,1)$. Find an estimator $\hat{\mu} : \bX^n \to \R$ such that
\begin{equation}
    \label{eq:vectormean}
    \Pr\left( \left\| \hat{\mu}\left(X_{1:n}^\eps\right) - \mu_P \right\| \leq \Phi_P(n, \alpha, \eps) \right) \geq 1 - \alpha
\end{equation}
with $\Phi_P(n, \alpha, \eps)$ as small as possible in terms of $n$, $\alpha$, $\eps$ and $P$.
\end{problem}

\begin{remark}
    In Problem \ref{problem:vectormean} one can always take $S$ to be the dual ball. But some norms can be achieved with much smaller sets, for instance, the infinity norm can be obtained as the supremum over $\{ v \in \R^d : \|v\|_1 = 1 \}$, but also as the supremum over $\{ \pm e_j : j \in [d] \}$.
\end{remark}

As noticed by Minsker \cite{minsker2018uniform}, the setting of Problem \ref{problem:vectormean} instantly suggests a natural family $\sF$ of functions associated with $\| \cdot \|$:
\begin{equation}
    \label{eq:sFnorm}
    \sF := \{ \langle v, \cdot \rangle : v \in S \}.
\end{equation}
The next lemma relates Problem \ref{problem:vectormean} to the problem of uniform mean estimation over the class $\sF$.

\begin{lemma}\label{lemma:relationnormmean} Let $\hat{E}_f$ be any given estimator for $f \in \sF$. Then, every function $\hat{\mu}:(\R^d)^n\to\R^d$ such that
\begin{equation*}
    \forall x_{1:n}\in(\R^d)^n\,:\,\hat{\mu}(x_{1:n})\in \argmin_{\mu \in \R^d} \left( \sup_{f \in \sF} \left| \hat{E}_f(x_{1:n}) - f(\mu) \right| \right)
\end{equation*}
satisfies
\[ \| \hat{\mu}(x_{1:n}) - \mu_P \|  \leq  2\sup_{f \in \sF}\left| \hat{E}_f(x_{1:n})  - f(\mu_P) \right|. \]
Moreover, given a mean estimator $\hat{\mu}$ for $\mu_P$, let $\hat{E}_f(x_{1:n}) = f(\hat{\mu}(x_{1:n}))$ for each $f\in \sF$, then
\[ \sup_{f \in \sF}\left| \hat{E}_f(x_{1:n})  - f(\mu_P) \right| = \| \hat{\mu}(x_{1:n}) - \mu_P \|. \]
\end{lemma}
\begin{proof} Just notice that
\begin{align*}
    \| \hat{\mu}(x_{1:n}) - \mu_P \| & = \sup_{f \in \sF} \left|f( \hat{\mu}(x_{1:n})) - f(\mu_P)\right| \\
    & \leq \sup_{f \in \sF} \left|f( \hat{\mu}(x_{1:n})) - \hat{E}_f(x_{1:n})\right| + \sup_{f \in \sF}\left| \hat{E}_f(x_{1:n})  - f(\mu_P) \right|\\
    & \leq \sup_{f \in \sF} \left|f( \mu_P) - \hat{E}_f(x_{1:n})\right| + \sup_{f \in \sF}\left| \hat{E}_f(x_{1:n})  - f(\mu_P) \right|\\
    & = 2\sup_{f \in \sF}\left| \hat{E}_f(x_{1:n})  - f(\mu_P) \right|
\end{align*}
where in the second inequality we used the definition of $\hat{\mu}$ as a minimizer. The second claim follows by definition.
\end{proof}

Let $\Gamma$ be the covariance matrix of $P$. Given $p > 2$ one can check that
\[  \nu_p(\sF,\Pm) = \sup_{v \in S} (\Pm\,|\langle X-\mu_\Pm, v \rangle|^{p})^\frac{1}{p} \quad\text{ and }\quad \underline{\sigma}_{\sF, P} = \inf_{v\in S} v \Gamma v^t. \]
In the special case where $\| \cdot \|$ is the euclidean norm, it holds that $\nu_2^2(\sF,\Pm) = \|\Gamma\|_\textrm{op}$. We also define the Gaussian width of a set $A \subset \R^d$ as
\[ w(A) := \Ex{\sup_{v \in A} \langle v, W \rangle}  \]
where $W \sim \sN(0, I_d)$. Notice that $\{G_Pf : f \in \sF\}$ has the same distribution as $\left\{ \left\langle v, \Gamma^\frac{1}{2} W  \right\rangle : v \in S \right\}$, thus
\[ \Ex Z(\sF) = \Ex \sup_{v\in S} \left\langle v, \Gamma^\frac{1}{2} W \right\rangle = w\left(\Gamma^\frac{1}{2}S\right). \]

Early work on Problem \ref{problem:vectormean} includes \cite{Minsker2015,Joly2017}. In the Hilbert space setting, a breakthrough result by Lugosi and Mendelson \cite{Lugosi2019a} presented a higher-dimensional version of the median of means. A more recent estimator by the same authors \cite{Lugosi2021}, based on high-dimensional Winsorized means, gives\footnote{They only derive explicit bounds for $p=2$, but the more general bounds follow from the same methods.}
\begin{equation*}\Phi_P(n,\alpha,\varepsilon)= C\,\left(\sqrt{\frac{{\rm tr}(\Gamma)}{n}} + \sqrt{\frac{\|\Gamma\|_{\rm op}}{n} \ln \frac{1}{\alpha}} + \inf_{p>1}\nu_p(\sF,\Pm)\eps^{1-\frac{1}{p}}\right)\end{equation*}
for some constant $C$. The term depending on $\eps$, here called contamination term, is optimal as it is optimal for single function mean estimation \cite{minsker2018uniform}. The term on $\alpha$, here called fluctuation term, is also known to be optimal if we only assume finite second moments. Finally, the independent term, here called complexity term, is also optimal.

The problem of general norms was studied by Lugosi and Mendelson \cite{Lugosi2019} and Depersin and Lecu\'{e} \cite{Depersin2021}, both papers present estimators based on the median of means. The best-known bound for Problem \ref{problem:vectormean} is given by \cite{OliveiraResende} using a trimmed mean based estimator and yields
\begin{equation}\label{eq:oliveirauniformeman}\Phi_P(n,\alpha,\varepsilon)=  C\,\left(\mathbb{E}\left\|\frac{1}{n}\sum_{i=1}^n X_i - \mu_P \right\| + \inf_{q \in [1,2]} \nu_q(\sF,P)\left(\frac{1}{n}\ln\frac{3}{\alpha}\right)^{1-\frac{1}{q}} + \inf_{p\geq 1} \nu_p(\sF,P)\eps^{1-\frac{1}{p}}\right)\end{equation}
requiring only $\nu_1(\sF, P) < \infty$.

There are also lower bounds available in the literature. In \cite{Depersin2021} a lower bound is obtained for Gaussian distributions under $\varepsilon=0$: let $\Gamma$ be positive-definite and let $P = \sN(\mu, \Gamma)$ in the setup of Problem \ref{problem:vectormean}, then exists $c>0$ such that
\begin{equation}\label{eq:lecuelowerbound}
    \Phi_P(n, \alpha, 0) \geq c\left( \frac{w\left(\Gamma^{\frac{1}{2}}S\right)}{\sqrt{n}} + \nu_2(\sF, P)\sqrt{\frac{1}{n}\ln \frac{1}{\alpha}}\right)
\end{equation}
for any choice of $\hat{\mu}$. Moreover, taking $\hat{\mu}$ as the empirical mean, \cite{Depersin2021} also proved that
\begin{equation}\label{eq:lecueupperbound}
    \Phi_P(n, \alpha, 0) \leq C\left( \frac{w\left(\Gamma^{\frac{1}{2}}S\right)}{\sqrt{n}} + \nu_2(\sF, P)\sqrt{\frac{1}{n}\ln \frac{1}{\alpha}}\right)
\end{equation}
for some constant $C>c$. When $\| \cdot \|$ is the euclidean norm and $P$ is any distribution with covariance $\Gamma$, all the complexity terms previously discussed are of the same order (see Proposition 2.5.1 of \cite{talagrand2014upper}), \textit{i.e.},
\[ \mathbb{E}\left\|\frac{1}{n}\sum_{i=1}^nX_i - \mu_P\right\|_2 \approx \sqrt{\frac{\text{tr}(\Gamma)}{n}} \approx \frac{w\left(\Gamma^{\frac{1}{2}}\mathbb{S}^{d-1}\right)}{\sqrt{n}}. \]
However, it is unclear if the expectation of the empirical process is close to the Gaussian width for general norms, which is optimal by \eqref{eq:lecuelowerbound}. In addition, \eqref{eq:lecuelowerbound} also points to the optimality of the fluctuation $\nu_2(\sF, P)\sqrt{\frac{1}{n}\ln \frac{1}{\alpha}}$ under finite second moment. Our main result in this section follows:

\begin{theorem}\label{thm:generalnorm} Assume an $\eps$-contaminated sample $X_{i:n}^\eps$ as in \S\ref{subsec:assumptions}. Recall the setup of Problem \ref{problem:vectormean} and the definition of $\Xi$ from Theorem \ref{thm:ga_vc}. Assume $\nu_p(\sF, P) < \infty$ for some $p > 2$. Let $K_n := K_n(\sF) = \lceil 12(d + 3) \ln n \rceil$ and suppose that \eqref{eq:condition_b1} holds taking $k := \lfloor \eps n\rfloor + \left( \lfloor \eps n \rfloor \vee K_n \right) < \frac{n}{2}$. Exist $c_1>0$ depending on $\nu_2(\sF,P)$ and on $\underline{\sigma}_{\sF, P}$ and $c_2 \geq \underline{\sigma}_{\sF, P}^{-1}$ depending only on $\underline{\sigma}_{\sF, P}$ such that if
\begin{equation}
    c_1 \sqrt{\nu_p} \left( \frac{K_n^{5-\frac{2}{p}}}{n^{1-\frac{2}{p}}} \right)^\frac{1}{4} + c_2 \nu_p \left( \eps^{1-\frac{1}{p}}\sqrt{n K_n } +  \left( \frac{K_n^{3-\frac{2}{p}}}{n^{1-\frac{2}{p}}} \right)^\frac{1}{2} \right) + \frac{2}{n} \leq \frac{\alpha}{2},
\end{equation}
then $\hat{\mu}_{n,k}(x_{1:n}) \in \argmin_{\mu \in \R^d} \left( \sup_{f \in \sF} \left| \Tmhat_{n,k}(f, x_{1:n}) - f(\mu)  \right| \right)$ satisfies \eqref{eq:vectormean} with
\[ \Phi_P(n, \alpha, \eps) =  \frac{w\left(\Gamma^{\frac{1}{2}}S\right)}{\sqrt{n}} + \nu_2(\sF)\sqrt{\frac{1}{n}\ln \frac{2}{\alpha}}. \]
\end{theorem}
\begin{proof} By Lemma \ref{lemma:relationnormmean},
\begin{equation}\label{eq:boundnormtm}
     \left\| \hat{\mu}_{n,k}(X^\eps_{1:n}) - \mu_P \right\|  \leq  2\sup_{f \in \sF}\left| \Tmhat_{n,k}^\eps(f)  - f(\mu_P) \right|.
\end{equation}

We use Theorem \ref{thm:ga_vc} to bound the RHS of \eqref{eq:boundnormtm}. Notice that the VC-subgraph dimension of $\sF$ is at most $d+2$. This is because for a given $f = \langle \cdot , v \rangle$ its subgraph $sg(f) := \{ (x,t) : \langle \cdot , v \rangle \leq t \} \in \R^{d+1}$ is a half-space and the VC dimension of all half-spaces is $d+2$, thus $vc(\sF) \leq d+2$. To apply Theorem \ref{thm:ga_vc} we need to bound the term depending on $\Xi$. Taking $w = \Gamma^\frac{1}{2}(v-v')$ we obtain
\[ \Xi(\delta) = \Ex \sup_{\substack{v,v' \in S,\\ \|\Gamma^\frac{1}{2}(v-v')\|_2 \leq \delta}} \langle \Gamma^\frac{1}{2}(v-v'), W \rangle \leq \Ex \sup_{\|w\|_2 \leq \delta} \langle w, W \rangle = \delta \Ex \|W\|_2 \leq \delta \sqrt{d} \leq \delta \sqrt{K_n}. \]
Thus, considering the cases $\lfloor \eps n \rfloor \leq K_n$ and $\lfloor \eps n \rfloor \geq K_n$ yields
\[ \frac{5\sqrt{K_n}}{2\underline{\sigma}_{\sF, P}}\,\Xi\left(14 \nu_p \left( \frac{K_n}{n} \vee \eps \right)^{\frac{1}{2}-\frac{1}{p}}\right) \leq \frac{35 \nu_p}{\underline{\sigma}_{\sF, P}} K_n  \left( \frac{K_n}{n} \vee \eps \right)^{\frac{1}{2}-\frac{1}{p}} \leq \frac{35 \nu_p}{\underline{\sigma}_{\sF, P}}\left( \eps^{1-\frac{1}{p}}\sqrt{n K_n } +  \left( \frac{K_n^{3-\frac{2}{p}}}{n^{1-\frac{2}{p}}} \right)^\frac{1}{2} \right). \]

Given the bound above and the assumptions of Theorem \ref{thm:generalnorm}, Theorem \ref{thm:ga_vc} yields
\[ \varrho = \sup_{\lambda \in \R}\left| \Pr\left( \sup_{f \in \sF}\left| \Tmhat_{n,k}^\eps(f)  - f(\mu_P) \right| \leq \lambda \right) -  \Pr\left( Z(\sF) \leq \lambda \right)
\right| \leq \frac{\alpha}{2}. \]

Now, recall that $Z(\sF)$ has the same distribution as $\sup_{v \in S} \langle \Gamma^\frac{1}{2}v, W \rangle$. The proof follows by taking $\lambda$ such that $\Pr\left( Z(\sF) \leq \lambda \right) \geq 1 - \frac{\alpha}{2}$. By Borell-TIS inequality, it suffices to take
\[\lambda = \frac{w(\Gamma^{\frac{1}{2}}S)}{\sqrt{n}} + \nu_2(\sF)\sqrt{\frac{1}{n}\ln \frac{2}{\alpha}}.\]
\end{proof}

When its assumption on $\alpha$ holds, Theorem \ref{thm:generalnorm} is optimal, as it matches the theoretical lower bound in \eqref{eq:lecuelowerbound}. Moreover, this is the first result to feature a complexity term governed by the Gaussian width for general norms. Methodologically, Theorem \ref{thm:generalnorm} departs fundamentally from prior work \cite{Minsker2015, Joly2017, Lugosi2019a, Lugosi2021, Lugosi2019, Depersin2021, OliveiraResende}. Rather than deriving uniform high-probability bounds for the supremum over $\sF$, our analysis tackles the problem by establishing a Gaussian approximation. Unlike previous results that allow a fixed contamination proportion $\eps \in [0,c)$, Theorem \ref{thm:generalnorm} requires $\eps$ to vanish. By Lemma \ref{lemma:epsbound}, this is a fundamental limitation of the Gaussian approximation approach.

We now discuss the presence of the assumption $\underline{\sigma}_{\sF, P} > 0$ in our bound. The case $\underline{\sigma}_{\sF, P} = 0$ corresponds to the existence of a non-trivial subspace $H \subset \R^d$ such that $X - \mu_P \in H^\perp$ a.s. Let $\sG = \{\langle \cdot, v\rangle : v \in S \cap H^\perp \}$, on the clean sample we have
\[ \sup_{f \in \sF}\left| \Tmhat_{n,k}(f, X_{1:n})  - f(\mu_P) \right| = \sup_{f \in \sG}\left| \Tmhat_{n,k}(f, X_{1:n})  - f(\mu_P) \right| \]
since $\langle X_i - \mu_P, v \rangle = 0$ for all $v \in H$. 
Replacing the previous equality in \eqref{eq:boundnormtm}, Theorem \ref{thm:generalnorm} follows considering only the subspace $H^\perp$. Thus, on the case $\eps = 0$, we can replace $\underline{\sigma}_{\sF, P}$ by the smallest non-zero eigenvalue of $\Gamma$. The same also holds when $\eps > 0$, but with a more convoluted argument, as discussed in Remark \ref{rem:replacesigmazero}. Consequently, since $\underline{\sigma}_{\sF, P}$ can always be replaced by the smallest non-zero eigenvalue, the assumptions on it will hold trivially for large $n$.

\section{High-dimensional bootstrap approximations}\label{sec:bootstrap}

Gaussian approximation results are typically accompanied by bootstrap procedures to enable feasible statistical inference. Let $\tilde{X}_{1:n}$ be an i.i.d. sample drawn from the empirical measure $\Pmhat_n = \frac{1}{n} \sum_{i=1}^n \delta_{X_i}$. Conditionally on a sample $X_{1:n}$, one can define the bootstrapped empirical process for each $f \in \sF$ as
\begin{equation}\label{def:em_bootstrap}
     \tilde{\G}_n(f) := \frac{1}{\sqrt{n}} \sum_{i=1}^n \left(f(\tilde{X}_i) - \Pmhat_n(f)\right).
\end{equation}
Setting $\tilde{Z}_n(\sF) := \sup_{f\in\sF} \tilde{\G}_n(f)$, bounds analogous to those discussed for $\varrho^E$ are available in the literature for the conditional Kolmogorov-Smirnov distance:
\[ \tilde{\varrho}^E := \sup_{\lambda \in \R } \left| \Pr\left( \left.\tilde{Z}_n(\sF) \leq \lambda \,\right| X_{1:n} \right) - \Pr\left( Z(\sF) \leq \lambda \right) \right|. \]

Similarly, we define an empirical bootstrap for trimmed mean as
\begin{equation}\label{def:tm_bootstrap}
     \tilde{\T}_{n,k}^\eps(f) := \frac{\sqrt{n}}{n-2k} \sum_{i=k+1}^{n-k} \left( f(\tilde{X}^\eps_{(i)}) - \Tmhat_{n,k}^\eps(f) \right),
\end{equation}
where $(\cdot)$ satisfies $f(\tilde{X}^\eps_{(1)}) \leq \cdots \leq f(\tilde{X}^\eps_{(n)})$. We also let $\tilde{Z}_{n,k}^\eps(\sF) := \sup_{f \in \sF}\, \tilde{\T}_{n,k}^\eps(f)$. Our goal is to establish that the conditional distribution of $\tilde{Z}_{n,k}^\eps(\sF)$ approximates the distribution of $Z(\sF)$, \textit{i.e.}, we aim to bound
\[ \tilde{\varrho} := \sup_{\lambda \in \R} \left| \Pr\left( \left.\tilde{Z}_{n,k}^\eps(\sF) \leq \lambda\,\right| X_{1:n}^\eps \right) - \Pr\left( Z(\sF) \leq \lambda \right) \right| \]
with high probability with respect to the contaminated sample $X_{1:n}^\eps$.

\begin{theorem}[High-dimensional bootstrap approximation for trimmed means; proof in \S\ref{sm:proof_bootstrap}]\label{thm:bootstrap_approximation} Assume an $\eps$-contaminated sample $X_{i:n}^\eps$ as in \S\ref{subsec:assumptions}. Let $nd \geq 3$ and $\nu_p := \nu_p(\sF, P) < \infty$ for some $p \in (2,\infty)$. If $k := 22\left(\left\lfloor \eps n \right\rfloor \vee \left\lceil \ln(nd) \right\rceil\right) < \frac{n}{2}$ and \eqref{eq:condition_b1} holds, then with probability at least $1 - c_1 \sqrt{\nu_p} \left( \frac{\ln^{5-\frac{2}{p}}(nd)}{n^{1-\frac{2}{p}}} \right)^\frac{1}{4} - \frac{1}{n}$ it holds that
\[\tilde{\varrho} \leq c_1 \sqrt{\nu_p} \left( \frac{\ln^{5-\frac{2}{p}}(nd)}{n^{1-\frac{2}{p}}} \right)^\frac{1}{4} + c_2 \nu_p \left( \eps^{1-\frac{1}{p}}\sqrt{n \ln(nd) } +  \left( \frac{\ln^{3-\frac{2}{p}}(nd)}{n^{1-\frac{2}{p}}} \right)^\frac{1}{2} \right) + \frac{1}{n}, \]
where $c_1 > 0$ depends on $\nu_2(\sF, P)$ and $\underline{\sigma}_{\sF, P}$, and $c_2 \geq \underline{\sigma}_{\sF, P}^{-1}$ depends only on $\underline{\sigma}_{\sF, P}$.
\end{theorem}

The convergence rate and the dependencies on $p$ and $\eps$ observed in Theorem \ref{thm:bootstrap_approximation} mirror the ones in Theorem \ref{thm:gaussian_approximation}. Along with that, the comparisons between Theorem \ref{thm:bootstrap_approximation} and the best-known result for the empirical mean when $\eps = 0$, which appear in \cite{chernozhuokov2022improved}, mirror the ones between Theorem \ref{thm:gaussian_approximation} and its analogous from \cite{chernozhuokov2022improved}. Lemma \ref{lemma:epsbound} also restricts the practical applicability of bounding $\tilde{\varrho}$: to establish valid bootstrap inference a bound on $\tilde{\varrho}$ needs to be combined with a corresponding Gaussian approximation (e.g., a bound on $\varrho$) via the triangle inequality. Therefore, since the Gaussian approximation breaks down when the contamination scales as shown by Lemma \ref{lemma:epsbound}, the overall bootstrap procedure is rendered unfeasible at that same contamination scale. Nonetheless, it remains an open question whether a direct bootstrap approximation --- one that bypasses the intermediate Gaussian approximation step --- is feasible under non-vanishing contamination levels. Specifically, it is unclear if 
\[ \sup_{\lambda \in \R} \left| \Pr\left( \left.\tilde{Z}_{n,k}^\eps(\sF) \leq \lambda\,\right| X_{1:n}^\eps
 \right) - \Pr\left( Z_{n,k}^\eps(\sF) \leq \lambda \right) \right| \to 0 \]
in the regime where $n,d\to\infty$ but $\eps \not\to 0$. % It is worth highlighting that in the uncontaminated setting ($\eps = 0$) with sub-exponential tails, \cite{deng2020beyond} demonstrated that direct bootstrap approximations for the empirical mean can attain a rate of $\sqrt{\frac{\ln^5 d}{n}}$, which improves upon the indirect bootstrap bounds of \cite{chernozhuokov2022improved}.

\subsection{Application to uniform confidence intervals}

We can apply Theorems \ref{thm:gaussian_approximation} and \ref{thm:bootstrap_approximation} to construct uniform confidence intervals. First, notice that if some $f\in\sF$ has a variance that dominates the variances of the other functions in $\sF$, then $G_P$ will likely be dominated by this single function $f$. To avoid this, we start by renormalizing every $f \in \sF$ to unit variance. The next Lemma, adapted from \cite[Lemma 3.1]{chernozhukov2023high}, provides a solution.

\begin{lemma} Let $\sF$ be a family with $d := |\sF| < \infty$. Let $S_n$ be any random variable such that $\sup_{\lambda \in \R} \left| \Pr\left( S_n \leq \lambda \right) - \Pr\left( Z(\sF) \leq \lambda \right) \right| \to 0$. If $S_n'$ is a random variable such that $\frac{\sqrt{\ln d}}{ \underline{\sigma}_{\sF, P} }| S_n - S_n' | \to 0$ in probability, then $\sup_{\lambda \in \R} \left| \Pr\left( S'_n \leq \lambda \right) - \Pr\left( Z(\sF) \leq \lambda \right) \right| \to 0$. The same holds under conditioning.
\end{lemma}

In particular, let $\tau \in [0,1]$ and $\widehat{\sigma}_{n,k}^\eps(f)$ be an estimate of the standard deviation of $f\in\sF$. Define
\[ S_n = \sqrt{n}\sup_{f\in\sF} \left( \sigma_f \right)^{-\tau} \left( \widehat{T}_{n,k}^\eps(f) - Pf \right) \quad \text{ and } \quad S_n' = \sqrt{n} \sup_{f\in\sF} \left( \widehat{\sigma}_{n,k}^\eps(f) \right)^{-\tau} \left( \widehat{T}_{n,k}^\eps(f) - Pf \right), \]
as well as their analogous $\tilde{S}_n$ and $\tilde{S}_n'$ replacing $\widehat{T}_{n,k}^\eps - Pf$ by $\tilde{T}_{n,k}^\eps - \widehat{T}_{n,k}^\eps$. It follows from Theorem \ref{thm:gaussian_approximation} and Theorem \ref{thm:bootstrap_approximation} that $S_n$, $S_n'$ and $\tilde{S}_n'$ all converge to $Z( \sF_\tau )$, where $\sF_\tau := \{ \sigma_f^{-\tau} f : f\in\sF \}$ (Theorems \ref{thm:gaussian_approximation} and \ref{thm:bootstrap_approximation} give the convergence of $S_n$ and $\tilde{S}_n$, and the Lemma above transfers it to the feasible statistics $S_n'$ and $\tilde{S}_n'$). Taking
\[ \widehat{\sigma}^\eps_{n,k}(f) = \sqrt{ \widehat{T}_{n,k}^\eps\left(  \left(f - \widehat{T}_{n,k}^\eps(f)\right)^2 \right) } \]
with the same $k$ as in Theorem \ref{thm:gaussian_approximation}, there exists a constant $c>0$ depending on $\tau$ and $\underline{\sigma}_{\sF,P}$\footnote{Apply the mean value theorem to the function $h(x) = x^{-\tau}$ for $\tau \in (0,1]$. For any $\sigma_f, \sigma_g \geq \frac{1}{2}\underline{\sigma}_{\sF,P}$, the absolute value of the derivative is bounded by $|h'(x)| = \tau x^{-\tau-1} \leq \tau \left(\frac{1}{2}\underline{\sigma}_{\sF,P}\right)^{-\tau-1}$. Therefore, $\left| \sigma_f^{-\tau} - \sigma_g^{-\tau} \right| \leq \tau \left(\frac{1}{2}\underline{\sigma}_{\sF,P}\right)^{-\tau-1} \left| \sigma_f - \sigma_g \right|$, which yields the desired constant $c = \tau \left(\frac{1}{2}\underline{\sigma}_{\sF,P}\right)^{-\tau-1}$. For $\tau = 0$, the difference is trivially $0$.} such that if $\sigma_f \geq \frac{1}{2}\underline{\sigma}_{\sF,P}$ for all $f\in\sF$, it holds that
\[ |S_n - S_n'| \leq c \sqrt{n}\sup_{f\in\sF} \left|  \widehat{T}_{n,k}^\eps(f) - Pf \right| \sup_{f\in\sF} \left|  \widehat{\sigma}^\eps_{n,k}(f) - \sigma_f \right| \]
and
\[ |\tilde{S}_n - \tilde{S}_n'| \leq c \sqrt{n} \sup_{f\in\sF} \left|  \tilde{T}_{n,k}^\eps(f) - \widehat{T}_{n,k}^\eps(f) \right| \sup_{f\in\sF} \left|  \widehat{\sigma}^\eps_{n,k}(f) - \sigma_f \right|. \]
In Lemma \ref{lem:bound_variance} we show that
\begin{equation}
\label{eq:bound_variance}
    \mathbb{P}\left( \sup_{f\in\sF} \left| \widehat{\sigma}^\eps_{n,k}(f) - \sigma_f \right| \geq \frac{c}{\underline{\sigma}_{\sF, P}} \left\{ \nu_{p \wedge 4}^2  \frac{\ln^{1 - \frac{2}{p \wedge 4}}(nd)}{n^{1-\frac{2}{p \wedge 4}}} + \nu_p^2 \eps^{1-\frac{2}{p}} \right\} \right) \leq \frac{1}{n}.
\end{equation}
Theorems \ref{thm:gaussian_approximation} and \ref{thm:bootstrap_approximation} can be used to approximate the quantities $\sqrt{n}\sup_{f\in\sF} \left|  \widehat{T}_{n,k}^\eps(f) - Pf \right|$ and $\sqrt{n} \sup_{f\in\sF} \left|  \tilde{T}_{n,k}^\eps(f) - \widehat{T}_{n,k}^\eps(f) \right|$ by $Z(\sF)$. Since these quantities involve the two-sided supremum $\sup_{f\in\sF}|\cdot|$, the relevant Gaussian object is $\sup_{f\in\sF\cup(-\sF)} G_P(f)$; throughout this subsection we therefore take $\sF$ to be symmetric (equivalently, replace $\sF$ by $\sF\cup(-\sF)$, which at most doubles $d$ and leaves all our bounds unchanged), so that $\sup_{f}|\cdot|$ is a genuine one-sided supremum and Theorems \ref{thm:gaussian_approximation} and \ref{thm:bootstrap_approximation} apply verbatim. By the Borell–TIS inequality, both quantities can be bounded by $2\sqrt{2\ln(nd)}$
with probability at least $1-\varrho - 2\tilde{\varrho} - \frac{1}{n}$. Combined with \eqref{eq:bound_variance}, this justifies the renormalization. Indeed, when compared with the order of $\varrho$ from Theorem \ref{thm:gaussian_approximation}, the additional error term is negligible. Notice that since $\tau \in [0,1]$, we allow not only for no standardization ($\tau = 0$), but also for partial standardization ($\tau \in (0,1)$) as done in \cite{lopes2020bootstrapping, liu2024robust}, and full standardization ($\tau = 1$).

Finally, we can construct the uniform confidence intervals. Let $\alpha \in (0, 1)$ and define the critical value $c_{1-\alpha}$ as the $(1-\alpha)$-quantile of the bootstrap statistic $\tilde{S}_n'$, or more specifically,
\[ c_{1-\alpha} := \inf \left\{ \lambda \in \R : \Pr\left( \tilde{S}_n' \leq \lambda \mid X_1, \dots, X_n \right) \geq 1 - \alpha \right\}. \]
This leads to the uniform confidence interval $\mathcal{C}_n$ for $Pf$
\[ \mathcal{C}_n(f) := \left[ \widehat{T}_{n,k}^\eps(f) - \frac{c_{1-\alpha} \widehat{\sigma}_{n,k}^\eps(f)^\tau}{\sqrt{n}}, \widehat{T}_{n,k}^\eps(f) + \frac{c_{1-\alpha} \widehat{\sigma}_{n,k}^\eps(f)^\tau}{\sqrt{n}} \right], \]
which satisfies $\Pr\left( \forall f \in \sF, Pf \in \mathcal{C}_n(f) \right) \geq 1 - \alpha - O(\varrho + \tilde\varrho)$.

\subsubsection{Experiments}

In this section, we evaluate the finite-sample performance of the proposed empirical bootstrap for the trimmed mean (TM, as in \eqref{def:tm_bootstrap}) in constructing uniform confidence intervals. We compare our approach against three natural baselines and state-of-the-art robust estimators: the standard empirical bootstrap for the empirical mean (EM, as in \eqref{def:em_bootstrap}), the approach proposed by \cite{liu2024robust} (MOM), and the approach proposed by \cite{kock2025high} (WM). We report the empirical coverage and average interval widths of the uniform confidence intervals at a confidence level of $1 - \alpha = 0.95$. All empirical results are averaged over $200$ independent Monte Carlo replications. We consider all four combinations of $n \in \{200, 1000\}$ and $d \in \{1000, 5000\}$. Throughout the simulations, the covariance matrix $\Sigma$ follows an autoregressive structure $\Sigma_{i,j} = 0.5^{|i-j|}$ for $1 \leq i, j \leq d$. The data generating mechanisms were adapted from \cite{liu2024robust} and are defined as follows:

\begin{enumerate}
    \item \textbf{Contaminated Case:} To evaluate adversarial robustness, we add outliers to a clean Gaussian sample with covariance $\Sigma$. We randomly select a subsample of size $\lfloor \eps n \rfloor$, where $\eps = 0.5 / \sqrt{n \ln d}$. For each contaminated observation, all $d$ entries are set to $10$.

    \item \textbf{Elliptical Case:} To introduce complex dependency structures while maintaining heavy tails, the data is generated from an elliptical distribution $X_i = \eta \Sigma^{1/2} Z_i / \|Z_i\|_2$, where $Z_i \sim \mathcal{N}(0, I_d)$. The random radius $\eta$ is defined as $\eta = \sqrt{\frac{dF}{\nu^2-2}}$, where $F$ follows an $F$-distribution with $\nu = 3$ degrees of freedom. 

    \item \textbf{Low-dimensional Case:} We generate data in dimension $50$ following the elliptical case and then, using an orthonormal basis, we map the low-dimensional distribution into $\R^d$. 

\end{enumerate}

To construct the uniform confidence intervals, we compute the estimators and their corresponding bootstrap quantiles over $B = 500$ bootstrap iterations. We set the trimming level as $k = \lfloor \epsilon n \rfloor + \lceil \ln(d) \rceil$ for both the TM and the WM. We implement the MOM estimator following the default configuration recommended by \cite{liu2024robust}. Specifically, we allocate a $10\%$ hold-out sample for variance and mean estimation and set the number of buckets as $\lceil \ln(n) \rceil$. Since \cite{liu2024robust} do not consider contamination we let $\lceil \ln(n) \rceil + 2\lfloor \eps n \rfloor$ in the contaminated case. The addition of $2\lfloor \eps n \rfloor$ term is justified by the experiments of \cite{OliveiraResende}. We use a partial standardization parameter of $\tau = 0.9$ for both the TM and the MOM.

\begin{figure}[ht]
    \centering
    \includegraphics[width=\linewidth]{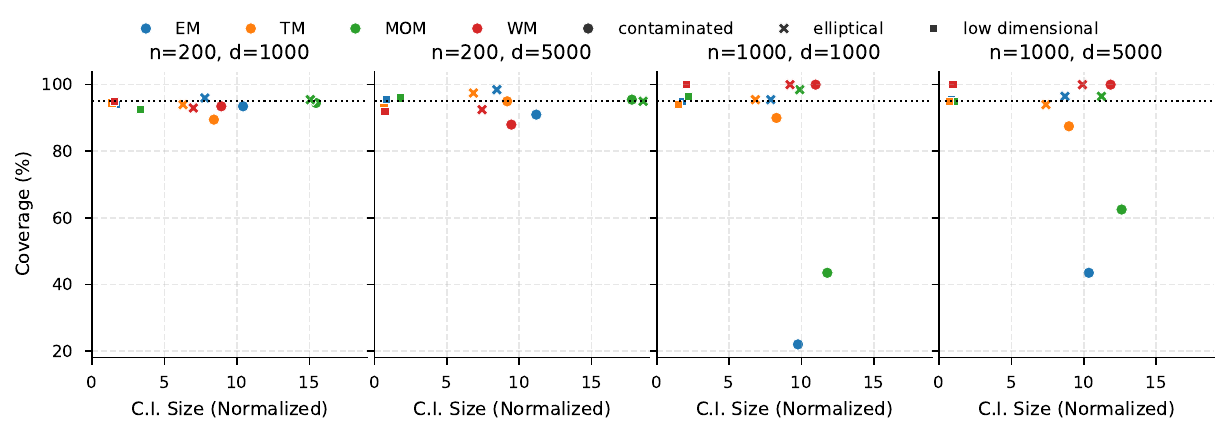}
    \caption{Comparison between the empirical bootstrap for the empirical mean (EM), the empirical bootstrap for the trimmed mean (TM), the bootstrap procedure from \cite{liu2024robust} and the bootstrap procedure proposed by \cite{kock2025high} (WM). We vary the dimension $d$, the sample size $n$ and the data distribution. We observe that the trimmed mean yields smaller confidence intervals while retaining a high coverage. The empirical average and the MOM approach do not behave well under contamination. Meanwhile, the approach from \cite{kock2025high} has a performance comparable to the trimmed mean, but favors coverage in exchange to larger confidence intervals.}
    \label{fig:comparisons}
\end{figure}

The results of our experiments are summarized in Figure~\ref{fig:comparisons}, which displays a comparison between the empirical bootstrap for the empirical mean (EM, as in \eqref{def:em_bootstrap}), our proposed empirical bootstrap for the trimmed mean (TM, as in \eqref{def:tm_bootstrap}), the MOM-based bootstrap procedure from \cite{liu2024robust}, and the winsorized mean (WM) bootstrap approach proposed by \cite{kock2025high}. Across the various sample sizes $n$, dimensions $d$, and data distributions, we observe that the proposed trimmed mean consistently yields the tightest confidence intervals while successfully retaining high coverage probabilities. In contrast, both the standard empirical average and the MOM approach fail to behave well under adversarial contamination. Meanwhile, the winsorized mean approach of \cite{kock2025high} performs comparable to our trimmed mean, but favors higher coverage at the expense of wider confidence intervals.

\begin{remark}[Computational Complexity]
The baseline bootstrap procedures, including the empirical mean, the MOM approach \cite{liu2024robust}, and the winsorized mean \cite{kock2025high}, generally require $O(B d n)$ operations. In contrast, our proposed method necessitates an element-wise ordering step, which introduces a slight computational overhead and results in a time complexity of $O(B d n \log n)$. This additional $\log n$ factor is a mild issue in practice. Since the primary focus is on the high-dimensional regime $d \gg n$, the dependency on $n$ is not a severe bottleneck. Empirically, the computational cost remains manageable: for $n=1000$ and $d=5000$, the MOM estimator takes approximately $1$ second to run, whereas our proposed estimator requires only around $10$ seconds.
\end{remark}

\section{Main proofs}\label{sec:proofideas}

Our proofs are based on a strategy introduced in \cite{OliveiraResende}. Start defining, for a given $M>0$, the truncation function 
\[
    \tau_M(x) := \begin{cases}M\text{, if }x>M\\x\text{, if }-M\leq x \leq M\\-M\text{, if }x<-M\end{cases}.
\]

The core observation from \cite{OliveiraResende} — which was inspired by \cite{Lugosi2021, rico2022a} — is that the trimmed mean of a contaminated sample is \textit{uniformly} close to the empirical mean of the truncated clean sample, provided that we carefully select the trimming level $k$ and the corresponding truncation level $M$. Specifically, 
\begin{equation}
    \label{eq:trimmingtruncationapprox}
    \sup_{f \in \sF}\left| \Tmhat_{n,k}^\eps(f) - \frac{1}{n}\sum_{i=1}^n \tau_M(f(X_i)) \right| \text{ is small.}
\end{equation}

This observation follows from two lemmata that will be introduced shortly. Before proceeding, we emphasize the utility of \eqref{eq:trimmingtruncationapprox}: it allows us to analyze the trimmed mean exploiting the rich literature on controlling the empirical mean of bounded random variables. In particular, we can apply the best-known Gaussian approximation result for the empirical mean, which is the following:

\begin{theorem}[Adapted from Theorem 2.4 and Lemma 4.3 of \cite{chernozhuokov2022improved}]\label{thm:ck_gaussian_approximation} Define
\[ \delta_{n,d,B} = \left(\frac{B^2 \ln^5(nd) }{n}\right)^\frac{1}{4} \text{and } \delta_{n,p,d,B} = \sqrt{\frac{B^2 \ln^{3-\frac{2}{p}}(nd)}{n^{1-\frac{2}{p}}}}. \]
Suppose that $\nu_4^4(\sF,P) \leq B^2 \nu_2^2(\sF,P)$ for every $f \in \sF$ for some $B>0$ and that $\underline{\sigma}_{\sF, P} > 0$. If $a,b > 0$ satisfy $\nu_2(\sF,P) \geq a$ and $\underline{\sigma}_{\sF, P} \geq b$, then exists $C := C(a,b) >0$ such that
    \begin{enumerate}
        \item if $f(X_1)-Pf$ is sub-exponential with $\psi_1$-Orlicz norm bounded by $B$ for all $f \in \sF$, then
        \[    \sup_{\lambda \in \R}  \left| \Pr\left( Z_n(\sF) \leq \lambda \right)
  - \Pr \left(Z(\sF) \leq \lambda\right)\right| \leq C \delta_{n,d,B};  \]
        \item if $\Ex{\max_{f \in \sF} |f(X_1)-Pf|^p} \leq B^p$ for some $p \in (2,\infty)$, then
        \[    \sup_{\lambda \in \R}  \left| \Pr \left( Z_n(\sF) \leq \lambda \right)
  - \Pr \left(Z(\sF) \leq \lambda\right)\right| \leq C \left(\delta_{n,d,B} \vee \delta_{n,p,d,B}\right).  \]
    \end{enumerate}
\end{theorem}

Of course, the issue remains in accurately bounding the "is small" term in \eqref{eq:trimmingtruncationapprox}, which will be addressed in the next two lemmas. First, let $\sG$ be a family of functions and $M > 0$, and define
\begin{eqnarray}\label{eq:defcountingclean} 
V_M(\sG) &:=& \sup_{g\in\sG} \sum_{i=1}^n\Ind\{|g(X_i)|>M\},
\end{eqnarray}
which counts the number of large values of $g$ for the worst-case function $g \in \sG$. The following lemma states that if $V_M(\sG)$ is small, then \eqref{eq:trimmingtruncationapprox} is justified.

\begin{lemma}[Bounding lemma \cite{OliveiraResende}]
\label{lem:bounding_simple} Let $\sG$ be a family of functions from $\bX$ to $\R$. %such that a measure $P$ is 1-compatible with $\sG_j$ for every $j=1, \cdots, m$. 
% Define
% \[\sG:=\sG_1 + \sG_2 + \dots + \sG_m = \{g_1+g_2+\dots + g_m\,:\, g_j\in\sG_j,\,j=1,2,\dots,m\}.\]
% % \[ \sG = \left\{ g_\alpha \right\}_{\alpha \in \sI}\text{, where }g_\alpha = \sum_{j=1}^m g_{j, \alpha}. \]
Also let $t\in\N$ and $M\geq 0$ be such that $V_{M}( \sG ) \leq t$. If $\phi$ satisfies $\phi n \in \N$ and
\[\frac{\lfloor \varepsilon n\rfloor + t}{n}\leq \phi <\frac{1}{2},\]
then
\[
    \sup_{g \in \sG}\, \left|\Tmhat_{n,\phi n}^{\varepsilon}\left(g\right) - \Pmhat_n\left(\tau_{M}\circ g\right)\right| \leq  6\phi M.
\]
\end{lemma}

Finally, we need to ensure that $V_M(\sG) \leq t$ with high probability. This is the purpose of the so-called ``counting lemma'' from \cite{OliveiraResende}. The original lemma in \cite{OliveiraResende} is somewhat convoluted and presents its results in terms of the Rademacher complexity of $\sF$. We here introduce the following simplified version of the original ``counting lemma'' from \cite{OliveiraResende}:

\begin{lemma}[Finite-dimensional counting lemma] \label{lemma:counting}
Let $d := |\sG| < \infty$, $M>0$ and $p \geq \Pr(|g(X_1)| > M)$ for all $g\in\sG$. For all $s \geq 10$,
\[ \Pr\left( V_M(\sG) \geq (1+s)np \right) \leq d e^{- np \frac{s \ln s}{2}}. \]
In particular, if $\nu_p(\sG, P) < \infty$,
\[ \Pr\left( V_M(\sG) \geq (1+s)n\frac{\nu_p^p}{M^p} \right) \leq d e^{- n\frac{\nu_p^p}{M^p} \frac{s \ln s}{2}}. \]
\end{lemma}
\begin{proof} Let $p_g = \Pr\left( |g(X)| > M \right)$. By Bennett's inequality for Bernoulli random variables,
\[ \Pr\left( \sum_{i=1}^n\Ind\{|g(X_i)|>M\} - np_g > snp_g \right) \leq s^{-\frac{snp_g}{2}}. \]
Markov's inequality yields $p_g \leq \frac{\nu_p^p}{M^p}$ for all $g \in \sG$. The result follows taking a union bound.
\end{proof}

\subsection{Proof of Gaussian approximation for finite-dimensional classes}\label{subsec:proof_gaussian}

\begin{proof}[Proof of Theorem \ref{thm:gaussian_approximation}] Since the trimmed mean is invariant by translation we assume, without loss of generality, that $Pf = 0$ for every $f \in \sF$, if it is not the case one can simply center the class $\sF$ and work with the centered class. We split the proof of Theorem \ref{thm:gaussian_approximation} in six steps:
\begin{enumerate}
\item \textit{Counting and bounding.} Taking $s = 10$ and $t = (1+s)n\frac{\nu_p^p}{M^p}$, Lemma \ref{lemma:counting} yields $V_M(\sF) \leq  t$ with high probability. Set
\[ k := \phi n \text{ where }\phi = \frac{\lfloor \eps n \rfloor+t}{n}, \]
we proceed using Lemma \ref{lem:bounding_simple} to obtain, for all $\lambda \in \R$,
\[ \Pr \left( Z_{n,k}^\eps(\sF) \leq \lambda \right) \leq \Pr\left( \sup_{f\in \sF} \frac{1}{\sqrt{n}}\sum_{i=1}^n \tau_M(f(X_i)) \leq \lambda + 6\phi M\sqrt{n} \right) + d e^{- 11n\frac{\nu_p^p}{M^p}}. \]

\item \textit{Approximate the truncated process.} We are almost ready to approximate $Z_n(\tau_M \circ \sF)$ by $Z(\tau_M \circ \sF)$ via Theorem \ref{thm:ck_gaussian_approximation}. Centering the class $\tau_M \circ \sF$ yields
\[ \Pr \left( Z_{n,k}^\eps(\sF) \leq \lambda \right) \leq \Pr\left( Z_n(\tau_M\circ \sF) \leq \lambda + \left(6\phi M + \sup_{f \in \sF} P(\tau_M \circ f )\right) \sqrt{n} \right) + d e^{- 11n\frac{\nu_p^p}{M^p}}. \]

Now, Theorem \ref{thm:ck_gaussian_approximation} yields, for all $\lambda \in \R$,
\begin{align*}
    \Pr \left( Z_{n,k}^\eps(\sF) \leq \lambda \right) \leq & \Pr\left( Z(\tau_M\circ \sF) \leq \lambda + \left(6\phi M + \sup_{f \in \sF} P(\tau_M \circ f )\right) \sqrt{n}  \right)\\
    &\, + C_1 \left(\frac{M^2 \ln^5(nd) }{n}\right)^\frac{1}{4} + d e^{- 11n\frac{\nu_p^p}{M^p}}
\end{align*}
where the constant $C_1$ depends on $\nu_2(\tau_M\circ \sF, P) \leq \nu_2(\sF, P)$ and on $\underline{\sigma}_{\tau_M\circ \sF, P}$. As will become clear in the end of the proof, \eqref{eq:condition_b1} implies $\underline{\sigma}^2_{\tau_M\circ \sF, P} \geq \frac{1}{2}\underline{\sigma}^2_{\sF, P}$ and we can say that $C_1$ depends only on $\nu_2(\sF, P)$ and $\underline{\sigma}_{\sF, P}$.

\item \textit{Use a Gaussian comparison inequality.} We now use the following lemma:
\begin{lemma}[Proposition 2.1 of \cite{chernozhuokov2022improved}]\label{lemma:gtog} Let $P,P'$ be distributions over $\bX$ and $\sG, \sG'$ be classes of real-valued functions. Let $d:=|\sG|=|\sG'|<\infty$. Assume $\underline{\sigma}_{\sG, P} > 0$, then
\[ \sup_{ \lambda \in \R } \left| \Pr\left(  Z(\sG, P) \leq \lambda \right) - \Pr\left( Z(\sG', P') \leq \lambda \right) \right| \leq C (\ln d) \sqrt{ \|\Sigma_{\sG,P}-\Sigma_{\sG',P'}\|_\infty } \]
for some constant $C$ depending only on $\underline{\sigma}_{\sG, P}$.
\end{lemma}
It yields
\[ \sup_{ \lambda \in \R } \left| \Pr\left(  Z(\sF, P) \leq \lambda \right) - \Pr\left( Z(\tau\circ\sF, P) \leq \lambda \right) \right| \leq C_2 (\ln d) \sqrt{ \|\Sigma_{\sF,P}-\Sigma_{\tau_M\circ\sF,P}) \|_\infty }, \]
where $C_2$ depends only on $\underline{\sigma}_{\sF, P}$. The RHS of the inequality above is bounded using the following lemma:

\begin{lemma}[Covariance bounds]\label{lemma:cov_bound} Let $\sF$ be a $P$-centered family and $M>0$, it holds that
\[ \left\|\Sigma_{\sF,P} - \Sigma_{\tau_M\circ\sF,P}\right\|_\infty \leq 4\nu_{p}^p M^{2-p}. \]
\end{lemma}
\begin{proof}
Given $f,g \in \sF$, we bound
\[ P\left(fg - (\tau_M \circ f)(\tau_M \circ g) \right) + P(\tau_M \circ f) P(\tau_M \circ g). \]
We apply Hölder's inequality to obtain
$P(\tau_M \circ f) = P\left(\tau_M \circ f - f\right) \leq P |f|\mathbf{1}_{|f|>M} \leq \nu_p^p M^{1-p}$. Using Hölder a couple more times yields:
\begin{align*}
     P\left|fg - (\tau_M \circ f)(\tau_M \circ g)\right| &\leq  M P\left|f-\tau_M \circ f\right|\mathbf{1}_{|f| > M} + M P\left|g-\tau_M \circ g\right|\mathbf{1}_{|g| > M}\\
     &\,\,\,\,\,\,+ P\left(|f|\mathbf{1}_{|f| > M} |g|\mathbf{1}_{|g| > M}\right)\\
     & \leq 2\nu_p^p M^{2-p}+ P\left(|f|\mathbf{1}_{|f| > M} |g|\mathbf{1}_{|g| > M}\right)\\
     & \leq 2\nu_p^p M^{2-p}+ \sup_{j\in[d]} P|f|^2\mathbf{1}_{|f| > M}\\
     & \leq 2\nu_p^p M^{2-p}+ \left(\nu_p^p\right)^\frac{2}{p}\left( \frac{\nu_p^p}{M^p} \right)^{1-\frac{2}{p}} \leq 3 \nu_p^p M^{2-p}.
\end{align*}
\end{proof}

We can now bound
\begin{align*}
    &\sup_{\lambda \in \R}\left|\Pr \left( Z_{n,k}^\eps(\sF) \leq \lambda \right) - \Pr\left( Z(\sF) \leq \lambda + \left(6\phi M + \sup_{f \in \sF} P(\tau_M \circ f )\right) \sqrt{n}  \right)\right|\\
    &\leq C_1 \left(\frac{M^2 \ln^5(nd) }{n}\right)^\frac{1}{4} + C_2(\ln d)\sqrt{4\nu_{p}^p M^{2-p}} + d e^{- 11n\frac{\nu_p^p}{M^p}}.
\end{align*}

\item \textit{Anti-concentration.} To get rid of the term $\left(6\phi M + \sup_{f \in \sF} P(\tau_M \circ f )\right) \sqrt{n} $ we bound
\begin{equation}
    \label{eq:boundPtmf}
    \sup_{f \in \sF} P(\tau_M \circ f ) \leq \nu_p^p M^{1-p}
\end{equation}
and make use of a Gaussian anti-concentration inequality. In this case we use Nazarov's inequality, which is due to \cite{nazarov2003maximal} and has a self-contained proof in \cite{chernozhukov2017detailed}:

\begin{lemma}[Nazarov's inequality]\label{lemma:nazarov} Let $\sG$ be a family of $P$-centered functions with $|\sG| = d$. If $\underline{\sigma}_{\sG, P} > 0$, then for all $\delta > 0$,
\[ \sup_{\lambda \in \R}\Pr\left( \lambda \leq Z(\sG) \leq \lambda + \delta \right) \leq \frac{\delta}{ \underline{\sigma}_{\sG, P} }\left(2 + \sqrt{2 \ln d} \right). \]
\end{lemma}

Using Lemma \ref{lemma:nazarov} and assuming $n\geq 3$ we bound $2 + \sqrt{2 \ln d} \leq \frac{5}{2}\sqrt{\ln(nd)}$ to get
\begin{align*}
    \varrho \leq & \left(6\phi M + \nu_p^p M^{1-p} \right) \frac{5\sqrt{n\ln(nd)}}{2\underline{\sigma}_{\sF, P}} +  C_1 \left(\frac{M^2 \ln^5(nd) }{n}\right)^\frac{1}{4} + C_2(\ln d)\sqrt{4\nu_{p}^p M^{2-p}} + d e^{- 11n\frac{\nu_p^p}{M^p}}.
\end{align*}

\item \textit{Take $M = \nu_p \left( \frac{\lfloor\eps n \rfloor}{11n} \vee \frac{ \lceil \ln(nd) \rceil }{11n} \right)^{-\frac{1}{p}}$.} It yields $t = \lfloor \eps n \rfloor \vee \lceil \ln(nd) \rceil$ and $d e^{- 11n\frac{\nu_p^p}{M^p}} \leq \frac{1}{n}$. Notice that, up to redefining $C_1$, we can bound $C_1 \left(\frac{M^2 \ln^5(nd) }{n}\right)^\frac{1}{4} \leq C_1 \sqrt{\nu_p} \left( \frac{\ln^{5-\frac{2}{p}}(nd)}{n^{1 - \frac{2}{p}}} \right)^\frac{1}{4}$. Now, consider $\lfloor \eps n \rfloor \leq \lceil \ln(nd) \rceil$ then, for $c_2 \geq \underline{\sigma}_{\sF, P}^{-1}$,
\[ \left(6\phi M + \nu_p^p M^{1-p} \right) \frac{5\sqrt{n\ln(nd)}}{2\underline{\sigma}_{\sF, P}} + C_2(\ln d)\sqrt{4\nu_{p}^p M^{2-p}} \leq c_2 \nu_p \left( \frac{\ln^{3-\frac{2}{p}}(nd)}{n^{1 - \frac{2}{p}}} \right)^\frac{1}{2}. \]
The case $\lfloor \eps n \rfloor \geq \lceil \ln(nd) \rceil$ follows using that $(\ln d)\sqrt{4\nu_{p}^p M^{2-p}} \leq  \sqrt{\eps n \ln(nd)}\sqrt{4\nu_{p}^p M^{2-p}}$ and collecting the terms.

\item \textit{Verify that \eqref{eq:condition_b1} implies $\underline{\sigma}^2_{\tau_M\circ \sF, P} \geq \frac{1}{2}\underline{\sigma}^2_{\sF, P}$.} It follows from Lemma \ref{lemma:cov_bound} and the choice of $M$ that
\[ \underline{\sigma}_{\tau_M\circ \sF, P}^2 \geq \underline{\sigma}_{\sF, P}^2 - 4\nu_p^pM^{2-p} = \underline{\sigma}_{\sF, P}^2 - 4 \nu_p^2 \left( \frac{\lfloor\eps n \rfloor}{11n} \vee \frac{ \lceil \ln(nd) \rceil }{11n} \right)^{1-\frac{2}{p}} \geq \underline{\sigma}_{\sF, P}^2 - 4 \nu_p^2 \left( \frac{k}{n} \right)^{1-\frac{2}{p}}, \]
and so $\underline{\sigma}^2_{\tau_M\circ \sF, P} \geq \frac{1}{2}\underline{\sigma}^2_{\sF, P}$ follows from \eqref{eq:condition_b1}.
\end{enumerate}
\end{proof}

\subsection{Proof of Gaussian approximation for VC classes}\label{subsec:proof_ga_vc}

Before diving into the proof of Theorem \ref{thm:ga_vc} we start with a few useful results.

\begin{lemma}[Talagrand's concentration inequality for VC-type classes, as in Theorem B.1 of \cite{chernozhukov2014anti}] \label{lemma:talagrandvctype} Let $\sG$ be a VC-type class with a constant envelope function $b$ and constants $A \geq e$, $v \geq 1$. Consider also that $\sigma^2$ satisfies $\nu_2^2(\sG) \leq \sigma^2 \leq b^2$. If $b^2 v \ln \frac{Ab}{\sigma} \leq n \sigma^2$, then for all $t \leq \frac{n \sigma^2}{b^2}$,
\[ \Pr\left( \sup_{g \in \sG} \left| \frac{1}{n}\sum_{i=1}^n g(X_i)- Pg \right| \geq C \sqrt{\frac{\sigma^2}{n} \left( t \vee \left(v \ln \frac{Ab}{\sigma}\right) \right)} \right) \leq e^{-t}, \]
where $C>0$ is an absolute constant.
\end{lemma}

We also use the following lemma relating truncated VC-subgraph to VC-type classes:

\begin{lemma}\label{lemma:vcclass}
    Assume $\sF$ is a VC-subgraph class and let $q \geq 1$. Then $\sF^q_M = \{ (\tau_M\circ f)^q : f \in \sF \}$ is VC-type with envelope $M^q$ and constants $(A,v) = (8e, 2\text{vc}(\sF))$.
\end{lemma}
\begin{proof}
First notice that $\text{vc}(\sF_M^q) \leq \text{vc}(\sF)$. Then use Theorem 5.11 of \cite{zhang2023mathematical}.
\end{proof}

And the following version of the counting lemma for VC-subgraph classes:

\begin{lemma}[Counting lemma for VC-subgraph classes]\label{lem:vc_counting}
Let $\sG$ be VC-subgraph class with dimension $v < \infty$. Given $M>0$, let $p = \sup_{g \in \sG} \Pr(|g(X_1)| > M)$. Then, for all $s \ge 10$,
\[ \Pr\left( V_M(\sG) > (1+4s)np \right) \leq 4 \left( \frac{en}{v+1} \right)^{2(v+1)} s^{-\frac{snp}{2}}. \]
In particular, if $\nu_p(\sG, P) < \infty$,
\[ \Pr\left( V_M(\sG) > (1+4s)n\frac{\nu_p^p}{M^p} \right) \leq 4 \left( \frac{en}{v+1} \right)^{2(v+1)} e^{- n\frac{\nu_p^p}{M^p} \frac{s \ln s}{2}}. \]
\end{lemma}
\begin{proof} Define the class of indicator functions associated with $\sG$ at threshold $M$:
\[ \sI_M := \left\{ x \mapsto \Ind\{|g(x)| > M\} : g \in \sG \right\} \]
Because $\sG$ is a VC-subgraph class with dimension $v$, the sets defined by $\{x: g(x) > M\}$ and $\{x: g(x) < -M\}$ are VC classes. Since $\sI_M$ represents the union of these sets, it is a VC class with dimension $v' \leq 2(v+1)$ \cite[Lemma 1]{depersin2024robust}.

We wish to bound the supremum of the empirical sum over $\sI_M$. We introduce an independent ghost sample $X'_1, \dots, X'_n$. By the classical symmetrization lemma, for any $t$ such that the variance condition $\sup_{h \in \sI_M} \Pr\left(\sum_{i=1}^n h(X'_i) - n \Ex h > \frac{t}{2}\right) \le \frac{1}{2}$ holds, one can bound \footnote{We verify this variance condition for the choice $t = 4snp$ made below. Since $h \in \sI_M$ is an indicator, $\sum_{i=1}^n h(X'_i)$ is a sum of Bernoulli variables with $n\Ex h \le np$, and by Markov's inequality $\Pr\left(\sum_{i=1}^n h(X'_i) - n\Ex h > \frac{t}{2}\right) \le \frac{n\Ex h}{t/2} \le \frac{np}{2snp} = \frac{1}{2s} \le \frac{1}{2}$ for all $s \ge 1$, so the condition holds.}
\[ \Pr\left( \sup_{h \in \sI_M} \sum_{i=1}^n (h(X_i) - \Ex h) > t \right) \leq 2 \Pr\left( \sup_{h \in \sI_M} \sum_{i=1}^n (h(X_i) - h(X'_i)) > \frac{t}{2} \right). \]

By the Sauer-Shelah lemma, the growth function satisfies $\Pi_{\sI_M}(2n) \leq \left( \frac{2en}{v'} \right)^{v'}$. Since all $h\in \sI_M$ are binary functions, there are at most $\Pi_{\sI_M}(2n)$ different values in the supremum, thus:
\[  \Pr\left( \sup_{h \in \sI_M} \sum_{i=1}^n (h(X_i) - \Ex h) > t \right) \leq 2 \left( \frac{2en}{v'} \right)^{v'} \sup_{h \in \sI_M} \Pr\left( \sum_{i=1}^n (h(X_i) - h(X'_i)) > \frac{t}{2} \right) \]

Taking $t = 4snp$ and using Bennett's inequality as in Lemma \ref{lemma:counting} yields
\[ \Pr\left( \sum_{i=1}^n (h(X_i) - h(X'_i)) > \frac{t}{2} \right) \leq 2 \Pr\left( \sum_{i=1}^n (h(X_i) - \Ex h) > \frac{t}{4} \right) \leq 2 s^{-\frac{snp}{2}}. \]
\end{proof}

\begin{proof}[Proof of Theorem \ref{thm:ga_vc}]
The overall argument is to use a $\delta$-net transforming the infinite-dimensional problem into a finite-dimensional problem and then controlling the errors. As done on the proof of Theorem \ref{thm:gaussian_approximation}, we assume without loss of generality that the class $\sF$ is centered with respect to the measure $P$. Recall that $K_n =  \lceil 12(\text{vc}(\sF) + 1) \ln n \rceil$. We split the proof into five steps:

\begin{enumerate}
    \item \textit{Use the counting and bounding lemmata.} Motivated by the proof of Theorem \ref{thm:gaussian_approximation} we let
    \[ M = \nu_p \left( \frac{\lfloor\eps n \rfloor}{41n} \vee \frac{ K_n }{41n} \right)^{-\frac{1}{p}} \text{ and } t = 41n\frac{\nu_p^p}{M^p} = \lfloor \eps n \rfloor \vee K_n. \]
    Lemma \ref{lem:vc_counting} and our choice of $K_n$ yields
    \[ \Pr\left( V_M(\sF) > 41n\frac{\nu_p^p}{M^p} \right) \leq 4 \left( \frac{en}{v+1} \right)^{2(v+1)} e^{- 11 n\frac{\nu_p^p}{M^p}} \leq \frac{1}{n}. \]
    Using Lemma \ref{lem:bounding_simple} one can take $k = \phi n$, $\phi = \frac{\lfloor \eps n \rfloor + t}{n}$ and bound
\[ \sup_{f \in \sF}\left|\T_{n,k}^\eps(f) - \frac{1}{\sqrt{n}}\sum_{i=1}^n \tau_M(f(X_i)) \right| \leq 6\phi M \sqrt{n}.\]
In addition, centering the class $\tau_M \circ \sF$ yields
\begin{equation}\label{eq:relationtrimtrunc}
\left|Z_{n,k}^\eps(\sF) - Z_n(\tau_M\circ\sF)\right| \leq \left(6\phi M + \nu_p^p M^{1-p} \right) \sqrt{n}.  
\end{equation}

\item \textit{Approximate by a $\delta$-net.} Let $\delta = 2\sqrt{\frac{K_n}{n}}$ and let $\sH_M$ be a $\delta M$-net of $\tau_M \circ \sF$, also let $\sH$ be the corresponding set in $\sF$ (notice that it may not be a $\delta M$-net for $\sF$), i.e.,
\[ \sH_M = \{ \tau_M \circ f :\, f \in \sH \}.   \]
By Lemma \ref{lemma:vcclass} we have $|\sH| = |\sH_M| \leq \left( \frac{8e}{\delta}  \right)^{2v}$. The approximating error induced by the $\delta$-net can be bounded by
\[ Z_n(\tau_M\circ\sF) - Z_n(\tau_M\circ\sH) \leq \sup_{\substack{f,g \in \sF \\ d_P(\tau_M \circ f,\tau_M \circ g) \leq \delta M}} \G_n( \tau_M\circ f-P\tau_M\circ f )-\G_n( \tau_M\circ g - P\tau_M\circ g ). \]
We now use Lemma \ref{lemma:talagrandvctype} with $\sigma = \delta M$ and $b = 2M$. Notice that $\tau_M\circ \sF \times \tau_M\circ \sF$ is VC-type with $A=8e$ and dimension $4v$ and the conditions
\[ 4 v \ln \frac{16e}{\delta} \leq \frac{n \delta^2}{4} \text{ and } \ln n \leq \frac{n\delta^2}{4}\]
hold. Thus, with probability at least $1-\frac{1}{n}$ and for some universal constant $C_1$,
\[ Z_n(\tau_M\circ\sF) - Z_n(\tau_M\circ\sH) \leq C_1 \delta M \sqrt{ \ln n \vee \left( 4v \ln \frac{16e}{\delta} \right) } \leq C_1\delta M \sqrt{K_n}. \]

\item \textit{Approximate $Z_n(\tau_M\circ\sH)$ by $Z(\sH)$.} Apply Theorem \ref{thm:ck_gaussian_approximation}, Lemma \ref{lemma:gtog} and Lemma \ref{lemma:cov_bound} to get:
\[ \sup_{\lambda \in \R} \left|\Pr\left(Z_n(\tau_M\circ\sH) \leq \lambda\right) - \Pr\left(Z(\sH) \leq \lambda\right) \right| \leq C_2 \left( \frac{M^2 K_n^5 }{n} \right)^\frac{1}{4} + C_3 K_n(\sF) \sqrt{\nu_{p}^p M^{2-p}} \]
for a constant $C_2$ that depends on $\underline{\sigma}_{\sF, P}$ and $\nu_2(\sF,P)$ and a constant $C_3$ that depends only on $\underline{\sigma}_{\sF, P}$. Notice that here we used condition \eqref{eq:condition_b1} to bound $\underline{\sigma}^2_{\tau_M \circ \sF, P} \geq \frac{1}{2} \underline{\sigma}^2_{\sF, P}$.

\item \textit{Approximate $Z_{n,k}^\eps(\sF)$ by $Z(\sF)$.} By the second and third steps together with Nazarov's inequality (Lemma \ref{lemma:nazarov}) we also have
\begin{align*}
    \sup_{\lambda \in \R} \left| \Pr\left( Z_{n,k}^\eps(\sF) \leq \lambda \right) - \Pr\left( Z(\sH) \leq \lambda \right) \right| \leq & \left((6\phi M + \nu_p^p M^{1-p})\sqrt{n} + C_1 M \delta \sqrt{K_n} \right) \frac{2 + \sqrt{2 K_n}}{\underline{\sigma}_{\sF, P}}\\
    & \, + C_2 \left( \frac{M^2 K_n^5 }{n} \right)^\frac{1}{4} + C_3 K_n \sqrt{\nu_{p}^p M^{2-p}} + \frac{1}{n}
\end{align*}

On the other hand, if $d_P(\tau_M \circ f, \tau_M \circ g) \leq \delta M$, then
\[ P(f-g)^2 = Pf^2 - 2Pfg + Pg^2 \leq d_P(\tau_M \circ f, \tau_M \circ g)^2 + 16 \nu_{p}^p M^{2-p} \leq \delta^2M^2 + 16 \nu_{p}^p M^{2-p}. \]

Thus the $d_P$ distance between a point $f \in \sF$ and its closest point $g \in \sH$ is at most $\sqrt{\delta^2 M^2 + 16 \nu_{p}^p M^{2-p}}$. Borell-TIS inequality gives
\[ \Pr\left( Z(\sF) - Z(\sH) \leq \Xi\left(\sqrt{ \delta^2M^2 + 16 \nu_{p}^p M^{2-p}}\right) + \sqrt{2 \left( \delta^2M^2 + 16 \nu_{p}^p M^{2-p}\right) \ln n } \right) \geq 1 - \frac{1}{n}. \]

Using Nazarov's inequality again and $2 + \sqrt{2 K_n} \leq \frac{5}{2} \sqrt{K_n}$ yields (up to redefining $C_1$)
\begin{align*}
    \sup_{\lambda \in \R} \left| \Pr\left( Z_{n,k}^\eps(\sF) \leq \lambda \right) - \Pr\left( Z(\sF) \leq \lambda \right) \right| \leq & \frac{5\sqrt{n K_n}}{2\underline{\sigma}_{\sF, P}}\left(6\phi M + \nu_p^p M^{1-p}\right) + C_1 K_n M\delta \\
    & \, + \frac{5\sqrt{K_n}}{2\underline{\sigma}_{\sF, P}}\,\Xi\left(\sqrt{ \delta^2 M^2 + 16 \nu_{p}^p M^{2-p}}\right) \\
    & \, + C_2 \left( \frac{M^2 K_n^5 }{n} \right)^\frac{1}{4} + C_3 K_n \sqrt{\nu_{p}^p M^{2-p}} + \frac{2}{n}
\end{align*}

\item \textit{Use the definitions of $M$ and $\delta$.} Start observing that $M =  \nu_p \left(\frac{t}{41n}\right)^{-\frac{1}{p}} \leq \nu_p \left(\frac{K_n}{41n}\right)^{-\frac{1}{p}}$ and so, up to redefining $C_1,C_2$, we have $C_1 K_n M\delta + C_2 \left( \frac{M^2 K_n^5 }{n} \right)^\frac{1}{4} \leq C_1\nu_p \left( \frac{K_n^{3-\frac{2}{p}}}{n^{1-\frac{2}{p}}} \right)^\frac{1}{2} + C_2 \sqrt{\nu_p} \left( \frac{K_n^{5-\frac{2}{p}}}{n^{1-\frac{2}{p}}} \right)^\frac{1}{4}$. If $\lfloor \eps n \rfloor \leq K_n$,
\begin{align*}
     & \frac{5\sqrt{n K_n}}{2\underline{\sigma}_{\sF, P}}\left(6\phi M + \nu_p^p M^{1-p}\right) + \frac{5\sqrt{K_n}}{2\underline{\sigma}_{\sF, P}}\,\Xi\left(\sqrt{ \delta^2 M^2 + 16 \nu_{p}^p M^{2-p}}\right) + C_3 K_n \sqrt{\nu_{p}^p M^{2-p}}\\ 
     \leq & \frac{5\sqrt{n K_n}}{2\underline{\sigma}_{\sF, P}}\left(12\frac{K_n}{n}M + \nu_p^p M^{1-p}\right) + \frac{5\sqrt{K_n}}{2\underline{\sigma}_{\sF, P}}\,\Xi\left(\sqrt{ 4\frac{K_n}{n} M^2 + 16 \nu_{p}^p M^{2-p}}\right) + C_3 K_n \sqrt{\nu_{p}^p M^{2-p}}\\ 
     \leq & \left( \frac{200}{\underline{\sigma}_{\sF, P}} + C_3 \right) \nu_p \left( \frac{K_n^{3-\frac{2}{p}}}{n^{1-\frac{2}{p}}} \right)^\frac{1}{2}   + \frac{5\sqrt{K_n}}{2\underline{\sigma}_{\sF, P}}\,\Xi\left(14 \nu_p \left( \frac{K_n}{n}\right)^{\frac{1}{2}-\frac{1}{p}}\right).
\end{align*}
If $\lfloor \eps n \rfloor \geq K_n$,
\begin{align*}
     & \frac{5\sqrt{n K_n}}{2\underline{\sigma}_{\sF, P}}\left(6\phi M + \nu_p^p M^{1-p}\right) + \frac{5\sqrt{K_n}}{2\underline{\sigma}_{\sF, P}}\,\Xi\left(\sqrt{ \delta^2 M^2 + 16 \nu_{p}^p M^{2-p}}\right) + C_3 K_n \sqrt{\nu_{p}^p M^{2-p}}\\ 
     \leq & \frac{5\sqrt{n K_n}}{2\underline{\sigma}_{\sF, P}}\left(12\eps M + \nu_p^p M^{1-p}\right) + \frac{5\sqrt{K_n}}{2\underline{\sigma}_{\sF, P}}\,\Xi\left(\sqrt{ 4\eps M^2 + 16 \nu_{p}^p M^{2-p}}\right) + C_3 \sqrt{\eps n K_n} \sqrt{\nu_{p}^p M^{2-p}}\\ 
     \leq & \left( \frac{200}{\underline{\sigma}_{\sF, P}} + C_3 \right) \nu_p \eps^{1-\frac{1}{p}}\sqrt{n K_n } + \frac{5\sqrt{K_n}}{2\underline{\sigma}_{\sF, P}}\,\Xi\left(14 \nu_p \eps^{\frac{1}{2}-\frac{1}{p}}\right).
\end{align*}
The final bound follows combining both cases and doing some overestimates.
\end{enumerate}
\end{proof}

\begin{remark}\label{rem:replacesigmazero} Keep assuming, without loss of generality, that $\sF$ is centered. It is possible to refine \eqref{eq:relationtrimtrunc} in the scenario where exists a symmetric set $\sG$ such that for all $f \in \sF$ either $f(X) = 0$ a.s. or $f(X) = g(X)$ a.s. for some $g \in \sG$. In this case we have
\begin{align*}
     Z_n(\tau_M\circ\sF) &= \sup_{f\in \sF}\frac{1}{\sqrt{n}}\sum_{i=1}^n \tau_M(f(X_i)) - P\tau_M\circ f \\ &= \sup_{g\in \sG}\frac{1}{\sqrt{n}}\sum_{i=1}^n \tau_M(g(X_i)) - P\tau_M\circ g \\ &=  Z_n(\tau_M\circ\sG).
\end{align*}
The proof follows replacing the family $\sF$ by the family $\sG$ in steps following \ref{eq:relationtrimtrunc}. In particular, the occurrences of the term $\underline{\sigma}_{\sF, P}$ can be replaced by $\underline{\sigma}_{\sG, P}$. This is specially useful in the context of vector mean estimation under general norms since if $X - \mu_P \in H^\perp$ a.s. we can replace (the uncentered) $\sF = \{ \langle \cdot, v \rangle : v \in S \}$ by (the uncentered) $\sG = \{ \langle \cdot, v \rangle : v \in S \cap H^\perp \}$. Notice that the centered versions of these classes satisfy
\[ \langle X, v \rangle - \langle \mu_P, v \rangle = \langle X - \mu_P, v \rangle = \langle X - \mu_P, \textrm{proj}_{H^\perp}(v) \rangle \]
and so either $\langle X - \mu_P, v \rangle = 0$ a.s. when $v \in H$ or $\langle X - \mu_P, v \rangle = \langle X - \mu_P, \textrm{proj}_{H^\perp}(v) \rangle$ a.s. when $v \not\in H$. Therefore, taking
\[ H = \left\{ v \in S : v^T\Gamma v = 0 \right\} \]
let us replace the dependence on  $\underline{\sigma}_{\sF, P}$ for a dependence on $\underline{\sigma}_{\sG, P}$, which is the smallest non-zero eigenvalue of $\Gamma$.
\end{remark}

\section*{Acknowledgments}
The author is very grateful to Roberto I. Oliveira for literature recommendations and several helpful discussions. He also thanks Guillaume Lecué for carefully reviewing the manuscript. The author used Cursor (with the GLM 5.2 and composer models) to review proofs and writing, mainly for keeping track of constants. It was also used to help generate code for the numerical experiments.

\section*{Funding}
The author was supported by Conselho Nacional de Desenvolvimento Científico e Tecnológico (CNPq).

\section{Supplementary Material}

\subsection{Proof of the upper bound on the contamination level}\label{sm:proof_epsbound}

\begin{proof}[Proof of Lemma \ref{lemma:epsbound}]
    Let $Y$ be a random vector with $d$ i.i.d. components with mean zero, variance $1$ and $L^p$-norm $\left(\Ex|Y_1|^p\right)^{1/p} = \frac{b_p}{\sigma}$. Let $B$ be an independent Bernoulli with probability $\frac{\eps}{2}$. Take $M\geq 1$ and define $P, P'$ as the distributions of
    \[ X = \sigma Y\quad\text{ and }\quad X' = (1-B)\sigma Y + B \sigma M \mathbf{1}_d. \]
    Notice that if $ \sum_{i=1}^n B_i \leq \eps n$ one can contaminate samples from $P'$ and $P$ to be indistinguishable. The probability that the two samples can be coupled in this way is bounded below using Chernoff's bound
    \[ \Pr\left( \sum_{i=1}^n B_i > \eps n \right) \leq e^{-\frac{\eps n}{6}}. \]
    
    We also have, writing $q := \frac{\eps}{2}$ for the mixing weight,
    \[ \Ex X = 0,\quad \Ex X' = q \sigma M\mathbf{1}_d,\quad \text{Cov}(X) = \sigma^2 I_d,\quad\text{and}\quad\text{Cov}(X') = (1-q) \sigma^2 I_d + q(1-q) \sigma^2 M^2 \mathbf{1}\mathbf{1}^T. \]

    Let $d_{KS}$ be the Kolmogorov-Smirnov distance. It holds that
    \begin{align*}
        & d_{KS}\left( \sqrt{n} \sup_f \widehat{E}_f^\eps(X_{1:n}^{\eps}) , \sup_f G_P(f) \right) + d_{KS}\left( \sqrt{n} \sup_f \left\{\widehat{E}_f^\eps(X_{1:n}^{'\eps}) - q \sigma M \right\} , \sup_f G_{P'}(f) \right)\\
        \geq & d_{KS}\left( \sqrt{n} \sup_f \widehat{E}_f^\eps(X_{1:n}^{\eps}) , \sup_f G_P(f) \right) + d_{KS}\left(  \sqrt{n} \sup_f \widehat{E}_f^\eps(X_{1:n}^{\eps}) , \sup_f G_{P'}(f) + \sqrt{n} q \sigma M \right) - e^{-\frac{\eps n}{6}}\\
        \geq & d_{KS}\left( \sup_f G_P(f) , \sup_f G_P(f) + \sqrt{n} q \sigma M \right) - d_{KS}\left( \sup_f G_P(f) , \sup_f G_{P'}(f) \right) - e^{-\frac{\eps n}{6}}\\
        \geq & d_{KS}\left( \sup_f G_P(f) , \sup_f G_P(f) + \sqrt{n} q \sigma M \right) - C(\sigma) \ln d \sqrt{\eps} \sigma M - e^{-\frac{\eps n}{6}},
    \end{align*}
    where in the last line we approximate $G_{P'}$ by $G_P$ using Lemma \ref{lemma:gtog}. Since $\text{Cov}(X') - \text{Cov}(X) = -q \sigma^2 I_d + q(1-q) \sigma^2 M^2 \mathbf{1}\mathbf{1}^T$ and each component of this matrix can be bounded in absolute value by $\eps \sigma^2 M^2$ (recall that $M \geq 1$, $q = \frac{\eps}{2}$ and $\eps < \frac{1}{2}$), the last inequality holds for some constant $C(\sigma)>0$ depending only on $\sigma$.

    Assume that $q M \geq \frac{\delta}{\sigma} (2 n \ln d)^{-\frac{1}{2}}$, as will be verified below. Since the Kolmogorov--Smirnov distance is non-increasing in the size of the shift between its two arguments,
    \begin{align*}
        d_{KS}\left( \sup_f G_P(f) , \sup_f G_P(f) + \sqrt{n} q \sigma M \right) & \geq  d_{KS}\left( \sqrt{2\ln d} \sup_f G_P(f) , \sqrt{2\ln d} \sup_f G_P(f) + \delta \right) \\
        & = \sup_{\lambda \in \R} \Pr\left( \lambda \leq \sqrt{2\ln d} \sup_f G_P(f) \leq \lambda + \delta \right).
    \end{align*}

    Recall that, by construction, $\sF$ is the family of the $d$ coordinate projections and the marginals of $P$ are i.i.d., so $\sup_f G_P(f)$ is the maximum of $d$ i.i.d.\ standard Gaussians. By the classical extreme-value theorem, $\sqrt{2\ln d}\,\sup_f G_P(f) - b_d$ converges in distribution to a standard Gumbel law $G$, where $b_d$ is the usual centering sequence. Moreover, since convergence in distribution to a continuous distribution on the line is uniform, there is a sequence $\eta_d \to 0$ such that, uniformly over $\lambda \in \R$,
    \[ \left| \Pr\left( \lambda \leq \sqrt{2\ln d}\,\sup_f G_P(f) \leq \lambda + \delta \right) - \Pr\left( \lambda \leq G \leq \lambda + \delta \right) \right| \leq \eta_d. \]
    Taking the supremum over $\lambda$ gives
    \[ \sup_{\lambda \in \R} \Pr\left( \lambda \leq \sqrt{2\ln d}\,\sup_f G_P(f) \leq \lambda + \delta \right) \geq \sup_{\lambda \in \R} \Pr\left( \lambda \leq G \leq \lambda + \delta \right) - \eta_d. \]
    Let $x(\delta) := \sup_{\lambda \in \R} \Pr\left( \lambda \leq G \leq \lambda + \delta \right) > 0$. Since $\eta_d \to 0$,
    \begin{align*}
        \liminf_{d\to\infty}\, d_{KS}\left( \sup_f G_P(f) , \sup_f G_P(f) + \sqrt{n} q \sigma M \right) \geq x(\delta) > 0.
    \end{align*}
    % In particular, for $d$ larger than some $d_1 = d_1(c,\sigma)$ we have $d_{KS}\left( \sup_f G_P(f) , \sup_f G_P(f) + \sqrt{n} q \sigma M \right) \geq \frac{1}{2} x(\sigma c)$.

    We now select $M$ so that $\nu_p^p(\sF, P)$ and $\nu_p^p(\sF, P')$ have the same order. Clearly $\nu_p^p(\sF, P) = b_p^p$. Taking $M = \frac{b_p}{\sigma} \eps^{-\frac{1}{p}} \geq 1$ yields
    \[ b_p^p \leq \nu_p^p(\sF, P') = (1-q) b_p^p + q (\sigma M)^p \leq 2b_p^p \]
    and 
    \[ \sigma^2 \leq \nu_2^2(\sF, P') = (1-q) \sigma^2 + q \sigma^2 M^2 \leq \sigma^2( 1 + q M^2 ) \leq 2\sigma^2 \]
    where we used that $q M^2 \leq \eps M^2 \leq 1$ (as required below).

    By the definition of $M$, $q M \geq \frac{\delta}{\sigma} (2 n \ln d)^{-\frac{1}{2}}$ holds as long as $\eps^{1-\frac{1}{p}} \geq \frac{2 \delta}{b_p}(2n \ln d)^{-\frac{1}{2}}$, which is an assumption. For large enough $d$,
    \begin{align*}
        & d_{KS}\left( \sqrt{n} \sup_f \widehat{E}_f^\eps(X_{1:n}^{\eps}) , \sup_f G_P(f) \right) + d_{KS}\left( \sqrt{n} \sup_f \left\{\widehat{E}_f^\eps(X_{1:n}^{'\eps}) - q \sigma M \right\} , \sup_f G_{P'}(f) \right)\\
        \geq & \frac{1}{2} x(\delta) - C(\sigma) \ln d \sqrt{\eps} \sigma M - e^{-\frac{\eps n}{6}}.
    \end{align*}

    If $\eps n \geq o_1 := 6 \ln \frac{8}{x(\delta)}$ and $\sqrt{\eps}M = \frac{b_p}{\sigma}\eps^{\frac{1}{2}-\frac{1}{p}} \leq 1 \wedge \frac{x(\delta)}{8 \sigma C(\sigma)\ln d}$ the previous term can be lower-bounded by $\frac{1}{4}x(\delta)$. We finish observing that for the sum of two distances to be larger than $\frac{1}{4}x(\delta)$ at least one distance must be larger than $c_1(\delta) := \frac{1}{8}x(\delta)$.
\end{proof}

\subsection{Order of Theorem 2.1 of \cite{chernozhukov2016empirical}} \label{sm:order_ga_vc}

The goal of this section is to show that \cite{chernozhukov2016empirical} yields a bound of order at least $\left( \frac{(K_n')^7}{n} \right)^\frac{1}{8}$. We recall the theorem below for convenience. The result from  \cite{chernozhukov2016empirical} was originally stated in terms of a Gaussian coupling, but the authors also provided an anti-concentration lemma to obtain a representation in Kolmogorov distance. Thus, to obtain a comparable result we combined its original form with the anti-concentration lemma the authors provided \cite[Lemma 2.2]{chernozhukov2016empirical}.

\begin{theorem}[Theorem 2.1 of \cite{chernozhukov2016empirical}]\label{thm:ck_ga_vc} Let $\sF$ be a family of functions and $P$ a probability measure. If
\begin{itemize}
    \item $\sF$ is VC-type with envelope $F$ and constants $A \geq 2$, $v \geq 1$,
    \item there exist constants $b \geq \sigma > 0$ and $p \in [4,\infty)$ such that $\nu_k^k \leq \sigma^2 b^{k-2}$ for $k=2,3,4$ and $\| F \|_{L^p(P)} \leq b$,
    \item and $K_n' := v \left( \ln n \vee \ln \frac{Ab}{\sigma} \right)$ satisfies $K_n' \leq n^\frac{1}{3}$,
\end{itemize}
then, for every $\gamma \in (0,1)$,
\[ \varrho \leq \Psi\left( C_1\left\{ \frac{b K_n'}{\gamma^\frac{1}{p} n^{\frac{1}{2}-\frac{1}{p}} } + \frac{(b \sigma^2 (K_n')^2)^\frac{1}{3}}{ \gamma^\frac{1}{3} n^\frac{1}{6} } \right\} \right) + C_2\left( \gamma + \frac{1}{n} \right) \]
where $C_1, C_2$ are positive constants depending only on $p$ and
\[ \Psi(\eta) =  \inf_{\delta, r > 0} \left\{ \frac{2}{\underline{\sigma}_{\sF, P}} \left( \eta + \Xi(\delta) + r \delta \right)\left( \sqrt{2 v  \ln \frac{Ab}{\delta} }+2\right) + e^{-\frac{r^2}{2}} \right\}. \]
\end{theorem}

If $\delta \leq n^\alpha$ for some $\alpha < 0$, then $\Psi(\eta) \geq C\frac{ \sqrt{ K_n' } }{ \underline{\sigma}_{\sF, P} } \eta$ for some constant $C$ depending on $\alpha$. Therefore, the bound given by the previous theorem has order at least
\[ \frac{(K_n')^\frac{3}{2}}{\gamma^\frac{1}{p} n^{\frac{1}{2}-\frac{1}{p}} } + \frac{(K_n')^\frac{7}{6}}{ \gamma^\frac{1}{3} n^\frac{1}{6} } + \gamma. \]
Since $p\geq4$, the previous equation is minimized when the last two terms are equal, yielding at least $\left( \frac{(K_n')^7}{n} \right)^\frac{1}{8}$.

\subsection{Proof of the bootstrap approximation result}\label{sm:proof_bootstrap}

Before starting the proof of our bootstrap approximation result we recall the following theorem from \cite{chernozhuokov2022improved}, which will be used in the proof:

\begin{theorem}[Adapted from Lemma 4.5 and Lemma 4.6 of \cite{chernozhuokov2022improved}]\label{lemma:ck_bootstrap_approximation}
    Make the same assumptions as in Theorem \ref{thm:ck_gaussian_approximation}. Let $\tilde{Z}_n$ be obtained from the Gaussian multiplier bootstrap or the empirical bootstrap. There exists a constant $C$ depending only on $\nu_2(\sF)$ and $\underline{\sigma}_{\sF, P}$ such that
    \begin{enumerate}
        \item if $f(X_1)-Pf$ is sub-exponential with $\psi_1$-Orlicz norm bounded by $B$ for all $f \in \sF$, then with probability at least $1- C \delta_{n,d,B}$,
        \[    \sup_{\lambda \in \R}  \left| \Pr \left( \left. \tilde{Z}_n(\sF) \leq \lambda \,\right| X_{1:n} \right)  - \Pr \left(Z(\sF) \leq \lambda \right)\right| \leq C\delta_{n,d,B};  \]
        \item if $\Ex{\max_{f \in \sF} |f(X_1)-Pf|^p} \leq B^p$ for some $p \in (2,\infty)$, then with probability at least
        
        $1 - C\left(\delta_{n,d,B} \vee \delta_{n,p,d,B}\right)$,
        \[    \sup_{\lambda \in \R}  \left| \Pr \left( \left. \tilde{Z}_n(\sF) \leq \lambda  \,\right| X_{1:n} \right)
      - \Pr \left(Z(\sF) \leq \lambda\right)\right| \leq C \left(\delta_{n,d,B} \vee \delta_{n,p,d,B}\right).  \]
    \end{enumerate}
\end{theorem}

\begin{proof}[Proof of Theorem \ref{thm:bootstrap_approximation}] As in the proof of Theorem \ref{thm:gaussian_approximation}, we assume, without loss of generality, the class $\sF$ is centered. Recall that
\[ \tilde{\T}_{n,k}^\eps(f) := \frac{\sqrt{n}}{n-2k} \sum_{i=k+1}^{n-k} f(\tilde{X}^\eps_{(i)}) - \Tmhat_{n,k}(f, X^\eps_{1:n}),\,\, (f \in \sF), \]
where the points $\tilde{X}^\eps_{i}$ are i.i.d. with law $\frac{1}{n}\sum_{i=1}^n\delta_{X_i^\eps}$. The strategy to prove the empirical bootstrap approximation is to approximate $\tilde{Z}_{n,k}^\eps(\sF)$ by $\tilde{Z}_{n}(\tau_M \circ \sF)$ and then use bootstrap approximation for the empirical mean. We let
\[ M = \nu_p \left( \frac{\lfloor\eps n \rfloor}{11n} \vee \frac{ \lceil \ln(nd) \rceil }{11n} \right)^{-\frac{1}{p}} \quad\text{ and }\quad t = 121 n\frac{\nu_p^p}{M^p} = 11\left( \lfloor \eps n \rfloor \vee \lceil\ln(nd) \rceil\right).  \]

\begin{enumerate}
\item \textit{Counting.} Define
\[ \tilde{V}_M(\sF) = \max_{f\in \sF} \sum_{i=1}^n \mathbf{1}_{\left\{\left|f(\tilde{X}_i)\right| > M \right\}} \]
to be the analogous of $V_M(\sF)$ for the empirical bootstrap. By Lemma \ref{lemma:counting} we know that the event $E = \left\{ V_M(\sF) < \frac{t}{11} \right\}$ satisfies $\Pr(E) \geq 1 - de^{-\frac{t}{11}} \geq 1 - \frac{1}{n}$. From now on we will be conditioning on $E$. Lemma \ref{lemma:counting} yields
\[ \Pr\left( \left. \tilde{V}_M(\sF) \geq t ~\right|~ E \right) \leq de^{-t} \leq \frac{1}{n}.\]

\item \textit{Controlling the contamination.} Let $\tilde{\eps}$ be the proportion of contaminated sample points in $\tilde{X}^\eps_{1:n}$, i.e.,
\[ \tilde{\eps} = \frac{|\{i: \tilde{X}_i^\eps \neq \tilde{X}_i\}| }{n}. \]
Notice that $\tilde{\eps}$ is independent of $X_{1:n}$. By Chernoff's bound for the binomial distribution: 
\[\Pr\left( \Tilde{\eps}n - \eps n \geq t \right) \leq \exp\left\{-\frac{\eps n}{3}\left(\frac{t}{\eps n}\right)^2 \right\} \leq \frac{1}{n}.\]

\item \textit{Bounding.} Now we can use the bounding lemma (Lemma \ref{lem:bounding_simple}) to approximate $\tilde{Z}_{n,k}^\eps(\sF)$ by $\tilde{Z}_{n}(\tau_M \circ \sF)$ with high probability conditioned on $E$. This is done using the lemma twice with $k = \phi n $ and $ \phi = 2\frac{t}{n}$. It yields
\begin{align*}
    \left| \tilde{\T}_{n,k}^\eps(f) - \tilde{\G}_{n}(\tau \circ f)\right| \leq & \sqrt{n} \left| \Tmhat_{n,k}(f, \tilde{X}^\eps_{1:n}) - \Pmhat_n(\tau_M \circ f, \tilde{X}_{1:n}) \right| \\
    & + \sqrt{n} \left| \Tmhat_{n,k}(f, X^\eps_{1:n}) - \Pmhat_n(\tau_M \circ f, X_{1:n}) \right|\\
    \leq & 12 \phi M \sqrt{n}.
\end{align*}

\item \textit{Approximating the empirical approximation.} Lemma \ref{lemma:ck_bootstrap_approximation} (case (i)) gives, with probability at least $1- C_1 \left(  \frac{M^2 \ln^5(nd)}{n} \right)^\frac{1}{4}$,
\[    \sup_{\lambda \in \R}  \left| \Pr \left( \left. \tilde{Z}_n(\tau_M\circ\sF) \leq \lambda \,\right| X_{1:n}\right) - \Pr \left(Z(\tau_M\circ\sF) \leq \lambda\right)\right| \leq C_1\left(  \frac{M^2 \ln^5(nd)}{n} \right)^\frac{1}{4}  \]
for some absolute constant $C_1$ depending only on $\nu_2(\tau_M\circ\sF) \leq \nu_2(\sF)$ and on $\underline{\sigma}_{\tau_M\circ\sF, P}$. The weak variance $\underline{\sigma}_{\tau_M\circ\sF, P}$ can be chosen such that $\underline{\sigma}_{\tau_M\circ\sF, P} \geq \frac{1}{2}\underline{\sigma}_{\sF, P}$, as a consequence of our final choice of $M$, of assumption \eqref{eq:condition_b1}, and of Lemma \ref{lemma:cov_bound}. 

\item \textit{Gaussian approximation and anti-concentration}. Lemma \ref{lemma:gtog} and Lemma \ref{lemma:cov_bound} yield
\[    \sup_{\lambda \in \R}  \left| \Pr \left( \left. \tilde{Z}_n(\tau_M\circ\sF) \leq \lambda \,\right| X_{1:n}\right) - \Pr \left(Z(\sF) \leq \lambda\right) \right| \leq C_1\left(  \frac{M^2 \ln^5(nd)}{n} \right)^\frac{1}{4} + C_2(\ln d)\sqrt{4\nu_{p}^p M^{2-p}} \]
for some $C_2$ that depends only on $\underline{\sigma}_{\sF,P}$. After combining with Nazarov's anti-concentration inequality (Lemma \ref{lemma:nazarov}) we get, with probability at least $1- C_1 \left(  \frac{M^2 \ln^5(nd)}{n} \right)^\frac{1}{4} - \frac{1}{n}$,
\[ \tilde{\varrho} \leq 12\phi M\frac{2 + \sqrt{2 \ln d}}{\underline{\sigma}_{\sF, P}}\sqrt{n} + C_1\left(  \frac{M^2 \ln^5(nd)}{n} \right)^\frac{1}{4} + C_2(\ln d)\sqrt{4\nu_{p}^p M^{2-p}} + \frac{2}{n}. \]

\item \textit{Finish the calculations}. The bound on $\tilde{\varrho}$ is the same, up to a constant, as the bound in $\varrho$ obtained in the proof of Theorem \ref{thm:gaussian_approximation}. Thus, the same bounds can be used. This finishes the proof.
\end{enumerate}
\end{proof}

\subsection{Proof of variance approximation}

\begin{lemma}
\label{lem:bound_variance}
Let $M, t, k$ be as defined in the proof of Theorem \ref{thm:gaussian_approximation}. Under the same assumptions, there exists an absolute constant $c > 0$ such that
\begin{equation}
    \mathbb{P}\left( \sup_{f\in\sF} \left| \widehat{\sigma}^\eps_{n,k}(f) - \sigma_f \right| \geq \frac{c}{\underline{\sigma}_{\sF, P}} \left\{ \nu_{p \wedge 4}^2  \frac{\ln^{1 - \frac{2}{p \wedge 4}}(nd)}{n^{1-\frac{2}{p \wedge 4}}} + \nu_p^2 \eps^{1-\frac{2}{p}} \right\} \right) \leq \frac{1}{n}.
\end{equation}
\end{lemma}

\begin{proof} First notice that $\widehat{\sigma}_{n,k}^\eps(f)$ is invariant by translation, so we assume, without loss of generality, that $\sF$ is centered. From the proof of Theorem \ref{thm:gaussian_approximation} we know that the event $E := \{ V_M(\sF) \leq t \}$ has probability at least $1 - \frac{1}{n}$ when taking
\[ M = \nu_p \left( \frac{\lfloor\eps n \rfloor}{11n} \vee \frac{ \lceil \ln(nd) \rceil }{11n} \right)^{-\frac{1}{p}} \quad\text{ and }\quad t = \lfloor \eps n \rfloor \vee \lceil \ln(nd) \rceil. \]
Recall also that $k = \lfloor \eps n \rfloor + t < \frac{n}{2}$.
If $E$ holds, $V_{2M}\left( \left\{ f - \widehat{T}_{n,k}^\eps(f) : f\in\sF  \right\} \right) < t$,  it follows from Lemma \ref{lem:bounding_simple} that
\begin{align*}
\left| \widehat{T}_{n,k}^\eps\left(  \left(f - \widehat{T}_{n,k}^\eps(f)\right)^2 \right) - \Pmhat_n\left( \left( \tau_{2M} \left(f - \widehat{T}_{n,k}^\eps(f)\right)\right)^2\right) \right| \leq 24 M^2 \frac{k}{n}.
\end{align*}

Now notice that $\tau_{2M}\left(f - \widehat{T}_{n,k}^\eps(f)\right) \neq \tau_M(f) - \widehat{T}_{n,k}^\eps(f)$ only when $|f|>M$. Given event $E$, the number of sample points for which $|f|>M$ is bounded by $t$, thus
\begin{align*}
\left| \Pmhat_n\left( \left( \tau_{2M} \left(f - \widehat{T}_{n,k}^\eps(f)\right)\right)^2\right) - \Pmhat_n\left( \left( \tau_Mf - \widehat{T}_{n,k}^\eps(f) \right)^2\right) \right| \leq \frac{t}{n} 4M^2
\end{align*}

Using Lemma \ref{lem:bounding_simple} can again be used to yield
\begin{align*}
\left| \Pmhat_n\left( \left( \tau_Mf - \widehat{T}_{n,k}^\eps(f) \right)^2\right) - \Pmhat_n\left( \left( \tau_Mf - \Pmhat_n(\tau_Mf) \right)^2\right) \right| \leq 24 M^2 \frac{k}{n} + 36 M^2 \left(\frac{k}{n}\right)^2 \leq 42 M^2 \frac{k}{n}.
\end{align*}

Lemma \ref{lemma:cov_bound} let us control the difference between the variance of $\tau_Mf$ and the one of $f$. Combining all these bounds, we have
\[ \sup_{f\in\sF} \left| \widehat{\sigma}^\eps_{n,k}(f)^2 - \sigma_f^2 \right| \leq \sup_{f\in \sF} \left| \Pmhat_n\left( \left( \tau_Mf - \Pmhat_n(\tau_Mf) \right)^2\right) - \sigma_{\tau_Mf}^2 \right| + 70 M^2 \frac{k}{n} + 4\nu_p^p M^{2-p} \]

To concentrate the empirical variance appearing above we can decompose
\[  \left| \Pmhat_n\left( \left( \tau_Mf - \Pmhat_n(\tau_Mf) \right)^2\right) - \sigma_{\tau_Mf}^2 \right| \leq \left| \Pmhat_n\left( \left( \tau_Mf \right)^2\right) - P(\tau_Mf)^2 \right| + \left| \left(\Pmhat_n\left( \tau_Mf \right) \right)^2 - \left( P\tau_Mf \right)^2 \right|  \]

For the first term, Bernstein's inequality implies that, with probability at least $1 - \frac{1}{2nd}$,
\[ \left| \Pmhat_n\left( \left( \tau_Mf \right)^2 \right) - P(\tau_Mf)^2 \right| \leq  \sqrt{\frac{2 P\left(\tau_Mf\right)^4 \ln(4nd)}{n}} + \frac{2 M^2 \ln(4nd)}{3n}. \]

If $p \in (2,4]$ we can use $1 \leq M \leq \nu_p \left( \frac{\lceil \ln(nd) \rceil}{11n} \right)^{-\frac{1}{p}}$ and $P(\tau_Mf)^4 \leq \nu_p^p M^{4-p}$ to bound, for some absolute constant $c$,
\begin{align*}
    \left| \Pmhat_n\left( \left( \tau_Mf \right)^2 \right) - P(\tau_Mf)^2 \right| & \leq c \nu_p^2  \frac{\ln^{1 - \frac{2}{p}}(nd)}{n^{1-\frac{2}{p}}} + \frac{2 M^2 \ln(4nd)}{3n}.
\end{align*}
Meanwhile, if $p>4$ we just bound $P(\tau_Mf)^4 \leq \nu_4^4$. Altogether,
\begin{align*}
    \left| \Pmhat_n\left( \left( \tau_Mf \right)^2 \right) - P(\tau_Mf)^2 \right| & \leq c \nu_{p \wedge 4}^2  \frac{\ln^{1 - \frac{2}{p \wedge 4}}(nd)}{n^{1-\frac{2}{p \wedge 4}}} + \frac{2 M^2 \ln(4nd)}{3n}.
\end{align*}

For the second term, we use $\left| a^2 - b^2 \right| = \left| a - b \right| \left| a + b \right| \leq \left| a - b \right| \left( \left| a - b \right| + 2\left| b \right| \right)$ to bound
\[  \left| \left(\Pmhat_n\left( \tau_Mf \right) \right)^2 - \left( P\tau_Mf \right)^2 \right| \leq \left| \Pmhat_n(\tau_Mf) - P(\tau_Mf) \right| \left( \left| \Pmhat_n(\tau_Mf) - P(\tau_Mf) \right| + 2\left| P(\tau_Mf) \right| \right). \]
Since we assumed $f$ is centered, we have $\left| P(\tau_Mf) \right| \leq \sigma_{\tau_Mf} \leq \sigma_f$. Similarly, for some absolute constant $c>0$, Bernstein's inequality yields, with probability at least $1 - \frac{1}{2nd}$,
\[ \left| \Pmhat_n(\tau_Mf) - P(\tau_Mf) \right| \leq \sqrt{\frac{2 \sigma_f^2 \ln(4nd)}{n}} + \frac{2 M \ln(4nd)}{3n}. \]

Thus, the first term is the dominating one. We can bound
\[ \left| \Pmhat_n\left( \left( \tau_Mf - \Pmhat_n(\tau_Mf) \right)^2\right) - \sigma_{\tau_Mf}^2 \right| \leq c \nu_{p \wedge 4}^2  \frac{\ln^{1 - \frac{2}{p \wedge 4}}(nd)}{n^{1-\frac{2}{p \wedge 4}}} + \frac{2 M^2 \ln(4nd)}{3n} \]
and a union bound yields, with probability at least $1-\frac{1}{n}$
\[ \sup_{f\in\sF} \left| \widehat{\sigma}^\eps_{n,k}(f)^2 - \sigma_f^2 \right| \leq c \nu_{p \wedge 4}^2  \frac{\ln^{1 - \frac{2}{p \wedge 4}}(nd)}{n^{1-\frac{2}{p \wedge 4}}} + \frac{2 M^2 \ln(4nd)}{3n} + 70 M^2 \frac{k}{n} + 4\nu_p^p M^{2-p}. \]

Since $M$ has order $\nu_p \left( \frac{k}{n} \right)^{-\frac{1}{p}}$ we bound, possibly updating the constant,
\[ \sup_{f\in\sF} \left| \widehat{\sigma}^\eps_{n,k}(f)^2 - \sigma_f^2 \right| \leq c \left\{ \nu_{p \wedge 4}^2  \frac{\ln^{1 - \frac{2}{p \wedge 4}}(nd)}{n^{1-\frac{2}{p \wedge 4}}} + \nu_p^2 \left( \frac{k}{n} \right)^{1-\frac{2}{p}} \right\}. \]

Considering the case $k$ is of order $\eps n$ and the case where $k$ is of order $\ln(nd)$ yields the final bound
\[ \sup_{f\in\sF} \left| \widehat{\sigma}^\eps_{n,k}(f)^2 - \sigma_f^2 \right| \leq c \left\{ \nu_{p \wedge 4}^2  \frac{\ln^{1 - \frac{2}{p \wedge 4}}(nd)}{n^{1-\frac{2}{p \wedge 4}}} + \nu_p^2 \eps^{1-\frac{2}{p}} \right\}. \]
The proof follows from the identity $| \widehat{\sigma}_{n,k}^\eps(f) - \sigma_f | \leq \underline{\sigma}_{\sF, P}^{-1} | \widehat{\sigma}_{n,k}^\eps(f)^2 - \sigma_f^2 |$.

\end{proof}

\bibliographystyle{unsrt}
\bibliography{bibliography}

\end{document}